\documentclass[reqno, 12pt]{amsart}

\usepackage[utf8]{inputenc}

\usepackage{tikz-cd}

\usepackage{amsfonts, amsthm, amsmath, amssymb}
\usepackage{hyperref}
\hypersetup{colorlinks=false}
\usepackage[all]{xy}

\usepackage{pifont}

\usepackage[margin=1in]{geometry}

\RequirePackage{mathrsfs} \let\mathcal\mathscr
  
\usepackage[colorinlistoftodos,french,bordercolor=black,backgroundcolor=red,linecolor=red,textsize=footnotesize,textwidth=0.75in,shadow]{todonotes}

\usepackage{hyperref}
\usepackage{upgreek}
\usepackage{mathabx,yfonts}
\usepackage{enumerate}

\numberwithin{equation}{section}

\newtheorem{theorem}{Theorem}[section] 
\newtheorem{lemma}[theorem]{Lemma}
\newtheorem{proposition}[theorem]{Proposition}

\theoremstyle{definition}
\newtheorem*{acknowledgements}{Acknowledgements}
\newtheorem{remark}[theorem]{Remark}

\newtheorem*{notation}{Notation and auxiliary estimates}

\newcommand{\e}{\mathrm{e}}

\renewcommand{\d}{\mathrm{d}}
\renewcommand{\phi}{\varphi}

\newcommand{\mbf}{\mathbf}

\newcommand{\bbQ}{\mathbb{Q}}
\newcommand{\bbP}{\mathbb{P}}

\newcommand{\bbR}{\mathbb{R}}

\newcommand{\bbC}{\mathbb{C}}
\newcommand{\bbZ}{\mathbb{Z}}
\newcommand{\bbA}{\mathbb{A}}

\newcommand{\eps}{\varepsilon}

\renewcommand{\leq}{\leqslant}

\renewcommand{\geq}{\geqslant}

\renewcommand{\c}{\mathbf{c}}

\renewcommand{\b}{\mathbf{b}}

\renewcommand{\r}{\mathbf{r}}

\DeclareMathOperator{\rk}{rk}

\DeclareMathOperator{\Spec}{Spec}
\DeclareMathOperator{\Gal}{Gal}

\let\emptyset\varnothing

\DeclareSymbolFont{bbold}{U}{bbold}{m}{n}
\DeclareSymbolFontAlphabet{\mathbbold}{bbold}

\newcommand{\md}[1]{  \left(\textnormal{mod}\ #1\right)}
 
\renewcommand{\P}{\mathbb{P}}
\newcommand{\Q}{\mathbb{Q}}

\newcommand{\N}{\mathbb{N}}
\newcommand{\R}{\mathbb{R}}

\newcommand{\Z}{\mathbb{Z}}
\renewcommand{\l}{\left}

\renewcommand{\r}{\right}
\renewcommand{\b}{\mathbf}
\renewcommand{\c}{\mathcal}
\renewcommand{\epsilon}{\varepsilon}

\renewcommand{\leq}{\leqslant}
\renewcommand{\geq}{\geqslant}

\title 
[The density of fibres with a rational point] 
{The density of fibres with a rational point for a fibration over hypersurfaces of low degree}

\author{Efthymios Sofos}
\address{The Mathematics and Statistics Building, University of Glasgow,
University Place, Glasgow, G12 8QQ, Scotland}
\email{efthymios.sofos@glasgow.ac.uk}

\author{Erik Visse-Martindale} 
\address{
Universiteit Leiden\\
Mathematisch Instituut, Niels Bohrweg 1, Leiden\\
2333 CA\\
Netherlands}
\email{h.d.visse@math.leidenuniv.nl}

\date{\today}

\begin{document}

\begin{abstract}
We prove asymptotics
for
the proportion of
fibres with a 
rational 
point
in a conic bundle fibration.
The
base of the fibration 
is a general hypersurface of low degree.
\end{abstract}

\maketitle

\setcounter{tocdepth}{1}
\tableofcontents

\section{Introduction}\label{s:intro}
Serre's problem~\cite{MR1075658} regards the density of elements in a family of varieties defined over $\bbQ$ that have a $\bbQ$-rational point. Special cases have been considered by Hooley~\cite{MR1199934,hoolb} Poonen--Voloch~\cite{poonenvoloch}, Sofos~\cite{MR3534973}, Browning--Loughran~\cite{arXiv:1705.01999}, and Loughran--Takloo-Bighash--Tanimoto~\cite{arXiv:1705.09244}. The recent investigation of Loughran~\cite{loughranjems} and Loughran--Smeets~\cite{MR3568035} provides an appropriate formulation of the problem and proves the conjectured upper bound in considerable generality.

Assume that $X$ is a variety over $\bbQ$ equipped with a dominant morphism $\upphi:X\to\bbP^n_\bbQ$. Letting $H$ denote the usual Weil height on $\bbP^n(\bbQ)$, Loughran and Smeets conjectured~\cite[Conj.1.6]{MR3568035} under suitable assumptions on $\upphi$, that 
for all large enough positive $t$, the cardinality of points
$b\in \bbP^n(\bbQ)$ with height $H(b)\leq t$ and such that  the fibre $\upphi^{-1}(b)$ has a point in $\bbR$ and $\bbQ_p$ for every prime $p$, has order of magnitude 
\begin{displaymath}
\frac{\#\{b\in \bbP^n(\bbQ):H(b)\leq t\}}{(\log t)^{\Delta(\upphi)}}
\end{displaymath}
for a non-negative quantity 
$\Delta(\upphi)$ that is defined in~\cite[Eq.(1.3)]{MR3568035}.

The cardinality of fibres of height 
at most 
$t$ and possessing a $\bbQ$-rational point is bounded by the quantity they considered, while the two quantities coincide if every fibre satisfies the Hasse principle. The problem of obtaining the conjectured lower bound for the number of fibres of bounded height with a $\bbQ$-rational point when $\upphi$ is general is considered rather hard because there is no general machinery for producing $\bbQ$-rational points on varieties.

There are only two instances in the literature of the subject where asymptotics have been proved unconditionally:
\begin{itemize}
\item the  base of the fibration is a toric variety (Loughran~\cite{loughranjems}),
\item the  base of the fibration is a wonderful compactification of an
adjoint semi-simple algebraic group (Loughran--Takloo-Bighash--Tanimoto \cite{arXiv:1705.09244}).
\end{itemize}

Our aim in this article is to extend the list above by proving asymptotics in a case of a rather different nature. The base of the fibration of our main theorem will be a generic hypersurface 
of large dimension compared to its degree. 

\subsection{The set-up of our results}\label{s:set-up}
Let $f_1$ and $f_2$ be homogeneous forms in $\bbZ[t_0,\ldots,t_{n-1}]$, of equal and even degree $d>0$ subject to some assumptions which are to follow.

We assume that both the projective varieties defined by $f_1(\mbf{t})=0$ and $f_2(\mbf{t})=0$ are smooth and irreducible. Moreover we assume that the variety defined by $f_1(\mbf{t})=f_2(\mbf{t})=0$ is a complete intersection. This is satisfied for generic $f_1$ and $f_2$ of fixed degree and in a fixed number of variables. The next condition is artificial in nature but its presence allows to adapt the arguments of Birch~\cite{birch} to our problem. Letting $\sigma(f_1,f_2)$ denote the dimension of the variety given by 
\begin{displaymath}
\rk\left(\frac{\partial f_i}{\partial t_j}\right)_
{0\leq j\leq n-1}^{1\leq i \leq 2}(\mbf{t}) \leq 1
\end{displaymath}
when considered as a subvariety in $\mathbb{A}_\bbC^n$, we shall demand the validity of 
\begin{equation}\label{ass:Birch}
n-\sigma(f_1,f_2)>3(d-1)2^{d}.
\end{equation} 

With more work along the lines of the present article, most of these assumptions may be removed. However, the assumption that $\deg(f_1)$ is even seems necessary and \eqref{ass:Birch} is vital for the entire strategy of the proof.  

\begin{remark}
We assume that the varieties defined by $f_i(\mbf{t})=0$ are smooth, so they are also irreducible since smooth hypersurfaces in $\bbP^{n-1}_\bbQ$ are irreducible if $n\geq 3$ holds. In particular we have $n>12$ by \eqref{ass:Birch}. 
\end{remark}

Let $B\subset \bbP^{n-1}_\bbQ$ be the hypersurface given by $f_2(\mbf{t})=0$. We recall that by the work of Birch~\cite{birch}, $B$ satisfies the Hasse principle, and moreover it satisfies weak approximation by work of Skinner~\cite{MR1446148}. From now on we also assume $B(\bbQ)\neq\emptyset$.

For every $i\in\{0,\ldots,n-1\}$ consider the subvariety $X_i$ of $\bbP^2_\bbQ\times\bbA^{n-1}_\bbQ$ defined by
\begin{align*}
\ x_0^2+x_1^2 = & f_1(t_0,\ldots,t_{i-1},1,t_{i+1},\ldots,t_{n-1})x_2^2,\\ 
& f_2(t_0,\ldots,t_{i-1},1,t_{i+1},\ldots,t_{n-1})=0.
\end{align*}
The maps $g_i:X_i\to B\subset\bbP^{n-1}_\bbQ$ sending a pair \begin{displaymath}
((x_0:x_1:x_2),(t_0,\ldots,t_{i-1},1,t_{i+1},\ldots,t_{n-1}))
\end{displaymath}
to $(t_0:\ldots:t_{i-1}:1:t_{i+1}:\ldots:t_{n-1})$ glue together, defining a conic bundle $X$ over the base $B$ -- this uses that $f_1$ has even degree. By assumption, $f_1$ is not a multiple of $f_2$, so the generic fibre of $X$ is smooth. 

If we were interested in counting $\bbQ$-rational points on $X$ then it would be necessary to make a further study into the equations defining a projective embedding of $X$ (as in~\cite[\S 2]{plms.12134}). Currently however, we are only interested in counting how many fibres of the conic bundle have a $\bbQ$-rational point. A \emph{conic bundle} is a 
 dominant morphism all of whose fibres are conics and 
whose generic
 fibre is  smooth. 
In this article we consider the conic bundle  
\begin{equation}\label{def:pibund} 
\upphi:X\to B
\end{equation} 
defined locally by $g_i$. We shall estimate asymptotically the probability with which the fibre $\upphi^{-1}(b)$ has a $\bbQ$-point as $b$ ranges over $B(\bbQ)$. For this, we define 
\begin{displaymath}
N(\upphi,t):= \#\big\{b\in B(\bbQ): H(b)\leq t,\ b\in \upphi(X(\bbQ))\big\},\ t\in \bbR_{>0},
\end{displaymath}
where $H$ is the usual naive Weil height on $\bbP^{n-1}(\bbQ)$. 

\begin{remark}
Since the degree of $f_1$ is even, the question if for a given $b\in B$ the fibre $\upphi^{-1}(b)$ contains a rational point is independent of a chosen representative.
\end{remark}

Consider the small quantity
\begin{equation}\label{ass:epsilonakid}
\eps_d:= \frac{1}{5 (d-1)2^{d+5}}.
\end{equation} 

\begin{theorem}\label{thm:main1}
In the set-up above there exists a constant $c_{\upphi}$ such that for $t\geq 2$ we have
\begin{displaymath}
N(\upphi,t)= c_{\upphi}\frac{t^{n-d}}{(\log t)^{\frac{1}{2}}}
+ O\Bigg(\frac{t^{n-d}}{(\log t)^{\frac{1}{2}+\eps_d}}\Bigg).
\end{displaymath}
\end{theorem} 

If $\upphi$ has a smooth fibre with a $\bbQ$-point then $c_\upphi$ is positive. This will be shown in Theorem \ref{thm:main2}, where we shall also provide an interpretation for the leading constant  $c_{\upphi}$. The proof of Theorem \ref{thm:main1} will be given in \S\ref{proof:girise}. 
The main idea is to feed sieve estimates
coming from the Rosser--Iwaniec half-dimensional sieve
into the major arcs of the 
 Birch 
circle method.

Theorem \ref{thm:main1} settles the first case in the literature of an asymptotic for the natural extension of Serre's problem to fibrations over a base that does not have the structure of a toric variety nor a wonderful compactification of an adjoint semi-simple algebraic group. Fibrations that have a base other than the projective space were also studied in the recent work of Browning and Loughran~\cite[\S 1.2.2]{arXiv:1705.01999}. In light of the work of Birch~\cite{birch}, our assumptions imply 
\begin{displaymath}
\#\big\{b\in B(\bbQ): H(b)\leq t\big\} \asymp t^{n-d}.
\end{displaymath}
A very special case of~\cite[Thm 1.4]{arXiv:1705.01999} proves $\lim_{t\to\infty}N(\upphi,t)/t^{n-d}= 0$, whereas Theorem \ref{thm:main1} provides asymptotics.

\subsection{The logarithmic exponent}
The power of $\log t$ occurring in our result is the one expected in the literature. Indeed, in the works of Loughran and 
Smeets~\cite[Eq.(1.4)]{MR3568035}, and Browning and Loughran~\cite[Eq.(1.3)]{arXiv:1705.01999}, one may find the expected power $\Delta(\upphi)$ defined as follows. For any $b\in B$ with residue field $\kappa(b)$, the fibre $X_b=\upphi^{-1}(b)$ is called \emph{pseudo-split} if every element of $\Gal(\overline{\kappa(b)}/\kappa(b))$ fixes some multiplicity-one irreducible component of $X_b\times\Spec(\overline{\kappa(b)})$. The fibre $X_b$ is called \emph{split} if it contains a multiplicity-one irreducible component that is also geometrically irreducible. Note that a split fibre is always pseudo-split and further note that for conic bundles these two notions are the same as the singular fibres are either double lines, or two lines intersecting. 

Now for every codimension one point $D\in B^{(1)}$ choose a finite group $\Gamma_D$ through which the action of $\Gal(\overline{\kappa(D)}/\kappa(D))$ on the irreducible components of $X_{\overline
{\kappa(D)}}$ factors. Let $\Gamma_D^\circ$ be the subset of elements of $\Gamma_D$ which fix some multiplicity one irreducible component. One sets $\delta_D=\#\Gamma_D^\circ/\#\Gamma_D$ and
\begin{displaymath}
\Delta(\upphi)=\sum_{D\in B^{(1)}}\big(1-\delta_D\big).
\end{displaymath}
By considering the possible singular fibres, it is clear that for a conic bundle, $\delta_D$ is different from $1$ if and only if $D$ is non-split.

In all the cases in the literature so far the power of $(\log t)^{-1}$ turns out to be $\Delta$. Indeed, this is also the case here. The only relevant codimension one point to take into account is $D:=Z(f_1,f_2)$; every other fibre is smooth and hence split. Suppose that $D$ is geometrically reducible, then the intersection between any two geometrically irreducible components lies in the singular locus of $D$, say $D^\textrm{sing}$. Being the intersection between varieties in projective space of codimension at most 2, its codimension is at most 4.

The affine cone above $D^\textrm{sing}$ is a subvariety of the affine variety defined by
\begin{displaymath}
\rk\left(\frac{\partial f_i}{\partial t_j}\right)_
{0\leq j\leq n-1}^{1\leq i \leq 2}(\mathbf{t}) \leq 1.
\end{displaymath}

As a subvariety, the affine cone over $D^\textrm{sing}$ is at most $\sigma(f_1,f_2)$, so its codimension is at least $n-\sigma(f_1,f_2)$. Hence the codimension of $D^\textrm{sing}$ in $\P^n_\Q$ is at least $n-\sigma(f_1,f_2)-1$. Hence we are led to an inequality 
\begin{displaymath}
4\geq n-\sigma(f_1,f_2)-1>3(d-1)2^d-1\geq 11,
\end{displaymath}
violating the combined assumptions~\eqref{ass:Birch} and $d\geq 2$. We conclude that $D$ is geometrically irreducible. 

The fibre above $D$ is given by $x_0^2+x_1^2=0$ over the function field $\kappa(D)$ and it is split if and only if $-1$ is a square in $\kappa(D)$. However, it is well known that the function field of a geometrically irreducible variety  
contains no non-trivial separable algebraic extensions of the base field. Since $-1$ is not a square in $\Q$, neither is it in $\kappa(D)$. Therefore, under the assumptions of Theorem \ref{thm:main1} we conclude that $\Delta(\upphi)=\delta_D=\tfrac{1}{2}$.

Alternatively, it was kindly remarked by the referee that 
one can prove that $D$ is geometrically integral 
by applying
 the 
Lefschetz hyperplane section theorem
to the hypersurface $f_1(\b t ) =0 $. Its divisor $D$ can only be reducible if the variety defined by $f_2 ( \b t )=0$ is also reducible,
which contradicts our assumptions on $f_2$.

 \begin{notation}
The symbol $\N$ will denote the set of strictly positive integers.
As usual, we denote the  
divisor, Euler and M\"{o}bius function by
$\tau$, $\phi$ and 
$\mu$.
We shall make frequent use of the estimates 
\begin{equation}\label{def:taurid}
\tau(m)\ll m^{\frac{1}{\log \log m}}
\end{equation} and 
\begin{equation}\label{def:phiqridten}
\phi(m) \gg m/ \log \log m
\end{equation} valid for all
integers
$m  \geq 3$
and found
in~\cite[Th.5.4]{MR3363366} and \cite[Th.5.6]{MR3363366} respectively.

We consider the forms $f_1$ and $f_2$ 
constant throughout our paper,
thus the implied constants in the Vinogradov/Landau notation 
$\ll, O(\cdot)$
are allowed to depend 
on $\upphi, f_1,f_2,n$ and $d$ 
without further mention.
Any dependence of the implied constants
on other parameters 
will be explicitly recorded
by the appropriate use of a subscript. 
For $z\in \mathbb{C}$ we let 
\[
\e(z):=\exp(2\pi i z)
.\]
The symbol
$v_p(m)$ will refer to the standard $p$-adic valuation of an integer $m$.
Lastly,
we shall
use the Ramanujan sum,
defined for $a\in \Z$ and $q\in \N$
as 
\begin{equation}\label
{def:ramanujsumdef}
c_q(a):=
\sum_{x\in (\Z/q\Z)^*}\e(ax/q )
.\end{equation}
\end{notation}
Denoting the indicator function of a set $A$ by $\b{1}_A$,
we have the following equality, 
\begin{equation}\label{eq:vonsternprimpowers}
c_{p^m}(a)=
p^{m-1}
\big(
p
\b{1}_{\{
v_p(a)\geq m
\}} 
-
\b{1}_{\{
v_p(a)\geq m-1
\}} \big),
(p \text{ prime}, a \in \Z,
m\geq 1)
.\end{equation}
Lastly, 
we shall make frequent use of the constant 
\begin{equation}
\label{def:c_0}
\c{C}_0
:=
\prod_{\substack{p \text{ prime } \\  p\equiv 3 \md{4}}}\Big(1-\frac{1}{p^2}\Big)^{1/2}
.\end{equation}

\begin{acknowledgements}  
This work 
started while Efthymios Sofos 
had a position at Leiden University.
It was completed while Erik Visse-Martindale
was visiting the Max Planck Institute in Bonn,
the hospitality of which
is greatly acknowledged.
The authors are 
very 
grateful to 
Daniel Loughran 
for useful
comments 
that helped improve the introduction 
and the end of~\S\ref{s:archdens}.
\end{acknowledgements}

\section{Using the Hardy--Littlewood circle method for Serre's problem}
\label{s:firststeps}
We begin by estimating the main quantity in Theorem~\ref{thm:main1}
by  
averages of an 
arithmetic function over a thin subset of integer vectors.
Let us first
define $\vartheta_\Q:\Z\to\{0,1\}$
as the indicator function of those integers $m$ 
such that the curve 
$x_0^2+x_1^2=mx_2^2 $
has a point over $\Q$.
For 
$P\in \R_{>0}$
we let 
\begin{equation}
\label{def:robinsoncrusoe}
\Theta_\Q(P)
:= 
\sum_{\substack{\b{x}\in \Z^n\cap P[-1,1]^n 
\\ f_1(\b{x})\neq 0,
f_2(\b{x})=0}} 
\vartheta_\Q(f_1(\b{x}))
.\end{equation}
In order to go from $\Q$-solutions to coprime $\Z$-solutions, 
we perform a standard M\"{o}bius transformation, 
where we cut off the range of summation at the price of an error term. 
This is the content of the following lemma. 
\begin{lemma}\label{s:inre} 
Under the assumptions of Theorem~\ref{thm:main1} we have for $t  \geq  1 $,
\[N(B,\upphi,t)
= 
\frac{1}{2}
\sum_{
l\in \N\cap [1, \log (  2t)]} \mu(l) \Theta_\Q(t/l) 
+O(t^{n-d} (\log 2 t)^{-1}).\]
\end{lemma}
\begin{proof}
For any
$b\in \P^{n-1}(\Q)$ 
there exists a unique, up to sign, 
$\b{y}\in \Z^n$ 
with $\gcd(y_0,\ldots,y_{n-1})=1$ and 
$b=[\pm \b{y}]$. 
Recalling that the degree of $f_1$ is even, 
allows to infer that 
the fibre $\upphi^{-1}(b)$ has a rational point 
if and only if 
$\vartheta_\Q(f_1(\b{y}))=1$,
hence
\[ N(B,\upphi,t)=
\frac{1}{2}
\#\{\b{y}\in \Z^n\cap t[-1,1]^n: \gcd(y_0,\ldots,y_{n-1})=1,
f_2(\b{y})=0, 
\vartheta_\Q(f_1(\b{y}))=1
\}
.\]
If $f_1(\b{y})=0$ then 
$\vartheta_\Q(f_1(\b{y}))=1$ (since $(0:0:1)$ is a point in $\upphi^{-1}([\b{y}])$)
and, therefore, the quantity above is 
\[ 
\frac{1}{2}
\sum_{\substack{
\b{y}\in \Z^n\cap t[-1,1]^n
\\
\gcd(y_0,\ldots,y_{n-1})=1 \\ f_2(\b{y})=0, 
f_1(\b{y})\neq 0  
}}  \vartheta_\Q(f_1(\b{y}))
+O(\#\{\b{y}\in \Z^n\cap [-t,t]^n:  
f_1(\b{y})=f_2(\b{y})=0\})
.\]
The assumption~\eqref{ass:Birch} allows to apply~\cite[Th.1,pg.260]{birch} 
with $R=2$ to immediately obtain
\[
\#\{\b{y}\in \Z^n\cap t[-1,1]^n:  
f_1(\b{y})=f_2(\b{y})=0\}
\ll t^{n-2d}, \ \ \ \ \ (t \geq 1 )
.\]
Thus we obtain 
equality with
\[ 
\frac{1}{2} 
\sum_{\substack{\b{y}\in \Z^n\cap t[-1,1]^n
\\ \gcd(y_0,\ldots,y_{n-1})=1 \\
f_1(\b{y})\neq 0,
f_2(\b{y})=0
}} \vartheta_\Q(f_1(\b{y}))
+O(t^{n-2d}).
\]
Using 
 M\"obius inversion and letting $\b{y}=l \b{x}$ we see that 
the sum over $\b{y}$ equals  
\[\sum_{\substack{\b{y}\in \Z^n\cap t[-1,1]^n\\
f_1(\b{y}) \neq 0,
f_2(\b{y})=0  }} \vartheta_\Q(f_1(\b{y}))
\sum_{\substack{l\in \N\\ l \mid \b{y}}} \mu(l)
=
\sum_{l\leq t} \mu(l)  \sum_{\substack{\b{x}\in \Z^n\cap \frac{t}{l}[-1,1]^n
\\
f_1(\b{x})\neq 0,
f_2(\b{x})=0 }} \vartheta_\Q(f_1(\b{x}))
,\] because 
$\vartheta_\Q(f_1(\b{y}))=\vartheta_\Q(f_1(\b{x}))$ holds 
due to $\deg(f_1)$ being even.
Hence 
\[N(B,\upphi,t)
= 
\frac{1}{2}
\sum_{
l\in \N\cap [1,t]} \mu(l) \Theta_\Q(t/l) 
+O(t^{n-2d}),\]
and now,
using that both $f_1$ and $f_2$ are smooth,~\eqref{ass:Birch} and~\cite[Th.1,pg.260]{birch} 
for $R=1$ yields
\[
|\Theta_\Q(t)|
\leq 
\#\{\b{y}\in \Z^n\cap t[-1,1]^n:  
f_2(\b{y})=0\}
\ll t^{n-d}
,\]
which shows that 
the collective contribution from large $l$ is
\[
\Big|
\sum_{
l\in \N\cap ( (\log 2 t)  ,t]} \mu(l) \Theta_\Q(t/l) 
\Big|
\ll \sum_{l> \log (2 t)   } (t/l)^{n-d}
\ll t^{n-d} \sum_{l >  \log (2 t) } l^{-2}
\ll t^{n-d} (\log 2t)^{-1}
,\]
where we used that $n-d\geq 2$ 
holds due to~\eqref{ass:Birch}.
\end{proof}
For
$m<0$
the curve $x_0^2+x_1^2=m x_2^2$ has no
$\R$-point, and therefore no $\Q$-point, hence $\vartheta_\Q(m)=0$.
Thus,
denoting
$
\max \{f_1([-1,1]^n
)\}:=
\max\{f_1(\b{t}):\b{t}\in [-1,1]^n\}
$, it is evident that
we have the equality
\[
\Theta_\Q(P)
=
\sum_{\substack{m\in \N
\\  m \leq  
\max \{f_1(
[-1,1]^n
)\}
P^d
}}
\hspace{-0,3cm}\vartheta_\Q(m)
\sum_{\substack{\b{x}\in \Z^n\cap P[-1,1]^n
\\ f_1(\b{x})=m, f_2(\b{x})=0}} 1
.\]
Writing 
$\mathrm{d}\boldsymbol{\alpha}$ for $\mathrm{d}\alpha_1
\mathrm{d}\alpha_2$ and 
using the identity 
\[
\int_{\boldsymbol{\alpha}\in [0,1)^2}
\e(\alpha_1(f_1(\b{x})-m)+\alpha_2 f_2(\b{x})
)
\mathrm{d}\boldsymbol{\alpha}
=
\begin{cases}  
1,&\mbox{if } f_1(\b{x})=m \text{ and } f_2(\b{x})=0,\\  
0,&\mbox{otherwise,} \end{cases}
\]
shows
the validity of
\begin{equation}\label{eq:appllmond}
\Theta_\Q(P)=
\int_{\boldsymbol{\alpha}\in [0,1)^2}
S(\boldsymbol{\alpha})
\overline{E_\Q(\alpha_1)}
\mathrm{d}\boldsymbol{\alpha}
,\end{equation}
where 
one uses the notation
\begin{equation}\label{def:sumfexp}
S(\boldsymbol{\alpha}):=
\sum_{\substack{\b{x}\in \Z^n\cap P[-1,1]^n
}} 
\e(\alpha_1 f_1(\b{x}) +\alpha_2 f_2(\b{x}))
\end{equation}
and 
\begin{equation}
\label{def:exponentrzer}
E_\Q(\alpha_1):=
\sum_{\substack{m\in \N \\   m \leq  
\max \{f_1([-1,1]^n)\}
P^d
}}
\vartheta_\Q(m)
\e(\alpha_1 m)
.\end{equation}
One has the obvious bound $E_\Q(\alpha_1)
\ll
P^d$
from the triangle inequality.
Recall the notation~\cite[pg.251,Eq.(4)-(7)]{birch},
that we repeat here for the convenience of the reader. 
For each $a_1,a_2,q$,
the interval $\mathcal{M}_{(a_1,a_2),q}(\theta)$ consists of those $\alpha\in[0,1]^2
$ satisfying
\[ 2|q\alpha_i - a_i|\leq P^{-d+2
(d-1)\theta}
\]
for all $i=1,2
$. For each $0<\theta\leq 1$ denote the set of `major arcs' by
\[
\mathcal{M}(\theta)=\bigcup_{1\leq q\leq P^{2
(d-1)\theta}}\bigcup_a \mathcal{M}_{(a_1,a_2),q}(\theta)
\]
where the second union is over those $a_1,a_2
$ satisfying both $\gcd(a_1,a_2
,q)=1$ and $0\leq a_i<q$ for all $i=1,2
$. 

Let us now deal with the complement of $\c M(\theta)$ that is usually referred to as the `minor arcs'.
In
our case the number of equations, denoted by $R$ in~\cite{birch}, satisfies 
$R=2$.
For 
small positive
$\theta_0$ and $\delta$ as in~\cite[pg.251,Eq.(10)-(11)]{birch},
that is
$1 > \delta + 16\theta_0$
and
$\frac{n-\sigma}{2^{d-1}}-6(d-1)>2\delta\theta_0^{-1}$
we have
\[\int_{\boldsymbol{\alpha}
\notin \c{M}(\theta_0)
}
\big|
S(\boldsymbol{\alpha})
\overline{E_\Q(\alpha_1)}
\big|
\mathrm{d}\boldsymbol{\alpha}
\leq 
\Big(
\int_{\boldsymbol{\alpha}
\notin \c{M}(\theta_0)
}
|
S(\boldsymbol{\alpha})
|
\mathrm{d}
\boldsymbol{\alpha}
\Big)
\max_{\alpha_1\in [0,1)}
|E_\Q(\alpha_1)|
,\]
hence,  
applying the result of~\cite[Lem.4.4]{birch}
on the first factor, and using the trivial bound
$E_\Q(\alpha_1) \ll P^d$
leads to
the following bound on the integral away from $\mathcal{M}(\theta_0)$:
\[\int_{\boldsymbol{\alpha}
\notin \c{M}(\theta_0)
}
\big|
S(\boldsymbol{\alpha})
\overline{E_\Q(\alpha_1)}
\big|
\mathrm{d}\boldsymbol{\alpha}
\ll
P^{n-d-\delta}  
.\]
By~\eqref{eq:appllmond}
this shows  
\[\Theta_\Q(P)=
\int_{\boldsymbol{\alpha}\in  \c{M}(\theta_0)}
S(\boldsymbol{\alpha})
\overline{E_\Q(\alpha_1)}
\mathrm{d}\boldsymbol{\alpha}
+O(P^{n-d-\delta}  ).
\]
Consistently modifying the setup, the following lemma is analogous to~\cite[Lem.4.5]{birch} and its proof is the same, using the notation introduced above. 
\begin{lemma}
\label{lem:asinbirch}
For any
$\theta_0,\delta$
satisfying~\cite[pg.251,Eq.(10)-(11)]{birch}
and 
under the assumptions of Theorem~\ref{thm:main1}
we have 
\[
\Theta_\Q(P)=\sum_{q\leq P^{2(d-1)\theta_0}}
\sum_{\substack{
\b{a}\in (\Z\cap [0,q))^2
\\
\gcd(a_1,a_2,q)=1
}}
\int_{\c{M}'_{\b{a},q}(\theta_0)}
S(\boldsymbol{\alpha})
\overline{E_\Q(\alpha_1)}
\mathrm{d}\boldsymbol{\alpha}
+O(P^{n-d-\delta})
,\]
where the modified set $\c{M}'_{\b{a},q}(\theta_0)$ is defined in~\cite[pg.253]{birch}
and consists of those $\alpha\in [0,1]^2$ satisfying $|q\alpha_i-a_i|\leq qP^{-d+2(d-1)\theta_0}$.
\end{lemma}
For $\b{a}\in (\Z\cap [0,q))^2$, write
\begin{equation}
\label
{def:singserr}
S_{\b{a},q} 
:=
\sum_{\substack{
\b{x}\in (\Z \cap [0,q))^n 
}}
\e\Big(\frac{a_1 f_1(\b{x})+
a_2 f_2(\b{x})
}{q}\Big)
\end{equation}
and
for $\boldsymbol{\Gamma} \in \R^2$ define
\begin{equation}
\label{def:singintegral}
I(\boldsymbol{\Gamma}):=
\int_{\boldsymbol{\zeta}\in [-1,1]^n}
\e 
( \Gamma_1 f_1(\boldsymbol{\zeta})+
\Gamma_2 f_2(\boldsymbol{\zeta})
 )
\d\boldsymbol{\zeta}
.\end{equation}
Recalling the notation
$\eta=2(d-1)\theta_0$
of~\cite[pg.254,Eq.(2)]{birch}, 
we now employ~\cite[Lem.5.1]{birch} with $\boldsymbol{\nu}=\b{0}$
to 
evaluate $S(\alpha)$ and to
see that under the assumptions of 
Lemma~\ref{lem:asinbirch}
we have  
\begin{align*}
&
\Theta_\Q(P)
-P^n
\sum_{q\leq P^{2(d-1)\theta_0}} q^{-n}
\sum_{\substack{
\b{a}\in (\Z\cap [0,q))^2
\\
\gcd(a_1,a_2,q)=1
}}S_{\b{a},q}
\int_{|\boldsymbol{\beta}|\leq P^{-d+\eta}}
I(P^d\boldsymbol{\beta})
\overline{E_\Q(\beta_1+a_1/q)}
\mathrm{d}\boldsymbol{\beta}
\\
\ll
&
P^{n-d-\delta}
+P^{n-1+2\eta}
\sum_{q\leq P^{\eta}}  
\sum_{\substack{
\b{a}\in (\Z\cap [0,q))^2
\\
\gcd(a_1,a_2,q)=1
}} 
\int_{|\boldsymbol{\beta}|\leq P^{-d+\eta}}
|E_\Q(\beta_1+a_1/q)|
\mathrm{d}\boldsymbol{\beta}
.\end{align*}
By
using
$E_\Q(\alpha)\ll P^d$ 
once more
we infer that the sum over $q$ 
in the error term
above
is 
\[
\ll
\sum_{q\leq P^{\eta}}  
q^2
 P^{2(-d+\eta)}
P^d
\ll P^{-d+5\eta}
,\]
hence we have proved the following lemma.
\begin{lemma}
\label{lem:nextinbirch}
Under the assumptions of Lemma~\ref{lem:asinbirch}
the quantity
$\Theta_\Q(P)
P^{-n+d}$ equals 
\[
\sum_{q\leq P^{\eta}} q^{-n}
\sum_{\substack{
\b{a}\in (\Z\cap [0,q))^2
\\
\gcd(a_1,a_2,q)=1
}}S_{\b{a},q}
\int_{|\boldsymbol{\beta}|\leq P^{-d+\eta}}
P^d I(P^d\boldsymbol{\beta})
\overline{E_\Q(\beta_1+a_1/q)}
\mathrm{d}\boldsymbol{\beta}
+O(P^{-\delta}+P^{-1+7\eta})
.\]
\end{lemma}

\section{Exponential sums 
with
terms detecting the existence of rational points
}
\label{s:exponentmultiplic}
As made clear by 
Lemma~\ref{lem:nextinbirch},
to verify
Theorem~\ref{thm:main1}
we need to
asymptotically  
estimate 
\[E_\Q\Big( \frac{a_1}{q}+\beta_1\Big)
=
\sum_{\substack{m\in \N \cap [1,T]
\\
x_0^2+x_1^2=m x_2^2
\text{ has a }\Q\text{-point}
}} 
\e^{2 \pi  i (\frac{a_1}{q}+\beta_1) m}
,\]
for integers $a_1,q$ and $\beta_1\in \R$ and $T=
\max\{f_1([-1,1]^n)\} P^d   $.
It suffices to first study the case $\beta_1=0$, and then apply Lemma~\ref{lem:easypiezy} at the end of this section.
To study $E_\Q(a_1/q)$
we shall 
rephrase 
the condition on $m$  
in a way that it only regards
the prime factorisation of $m$
and then use the Rosser--Iwaniec sieve.

We begin by alluding to
the formulas
regarding  
Hilbert symbols 
in~\cite[Ch.III,Th.1]{serrecourse},
which show that  
for 
strictly positive
integers $m$  
one has  
\begin{equation}
\label{def:riden}
\vartheta_\Q(m)=
\begin{cases}   
1,&\mbox{if }p\equiv 3 \md{4}\Rightarrow v_p(m)\equiv 0 \md{2},\\  
0,&\mbox{otherwise}. \end{cases}
\end{equation} 
Indeed, for
$m\in \Z_{>0}$, 
the curve 
$x_0^2+x_1^2=m x_2^2$
defines a smooth conic in $\P^2$ with an $\R$-point
and the Hasse principle combined with 
Hilbert's product formula~\cite[Ch.III,Th.3]{serrecourse}
proves~\eqref{def:riden}.
The function in~\eqref{def:riden}
is the characteristic function
of those integers $m$
that are sums of two integral squares, see~\cite[\S 4.8]{MR3363366}.
Landau~\cite[Eq.(4.90)]{MR3363366}  
proved the following asymptotic: 
\begin{equation}
\label{eq:landaurzero}
\sum_{1\leq m\leq x} \vartheta_\Q(m)
=
\frac{1}{
2^{1/2}
\c{C}_0 
}
\frac{x}{(\log x)^{1/2}}
+O\bigg(\frac{x}{(\log x)^{3/2}}\bigg)
, 
x
\in \R_{>1},
\end{equation} 
but this is not sufficient for us, since we will need a similar result restricted to those $m$ in an arithmetic progression.
Observe that 
the following holds 
due to periodicity,
\[
E_\Q\Big(\frac{a_1}{q}\Big)=\sum_{\substack{m\in \Z \cap [1,T]
\\
x_0^2+x_1^2=m x_2^2
\text{ has a }\Q\text{-point}
}} 
\hspace{-0,5cm}
\e^{2 \pi i \frac{a_1}{q} m}
= 
\sum_{\ell \in \Z\cap [0,q)}
\e(a_1 \ell/q)
\sum_{\substack{1\leq m \leq T\\ 
m\equiv \ell \md{q}}}
\vartheta_\Q(m)
.\]
The work of Rieger~\cite[Satz 1]{mr0174533}
could
now be 
invoked
to study the sum over $m\equiv \ell \md{q}$
when $\gcd(\ell,q)=1$.
One could attempt to use this 
to get asymptotic formulas 
for the cases with  $\gcd(\ell,q)>1$, however, 
we found it more straightforward to work instead with 
the function $\varpi$
in place of 
 $\vartheta_\Q$.  
This function $\varpi:\Z_{>0}\to \{0,1\}$
is defined as     
\begin{equation}
\label{def:biden}
\varpi(m):=
\begin{cases}  
1,&\mbox{if }p\mid m\Rightarrow p\equiv 1 \md{4},\\  
0,&\mbox{otherwise.} \end{cases}
\end{equation}
It is obvious that for all 
$m,k \in \Z_{>0}$ we have 
\begin{equation}
\label{eq:fulmul}
\varpi(mk)=\varpi(m)\varpi(k)
,\end{equation}
while a similar property does not hold for $\vartheta_\Q$ (to see this take $m=k=p$, where $p$ is any prime 
which is
$3\md{4}$).
This is the reason for choosing to work with $\varpi$ rather than $\vartheta_\Q$.
Our next lemma 
shows how one can replace $\vartheta_\Q$ by $\varpi$,
while simultaneously restricting the summation at the price of an error term.
\begin{lemma}
\label{lem:exponbrzer}
For 
$x,u\in \R_{\geq 1},
q\in \Z_{>0},
a_1\in \Z\cap [0,q)$ 
we have 
\[
\sum_{1\leq m \leq x} \vartheta_\Q(m) \e(a_1 m/q)=
\sum_{\substack{(k,t)\in  \Z_{>0}  \times \Z_{\geq 0} \\
2^t k^2 \leq u 
\\
p\mid k \Rightarrow p\equiv 3 \md{4}}}
\sum_{\ell \in \Z\cap [0,q)
}
\e(a_1 \ell/q)
\sum_{\substack{r\in \Z_{>0} \\ 
2^t k^2 r \equiv \ell \md{q}
\\
1\leq  r \leq x 2^{-t}k^{-2}
}}
\varpi(r)
+O\bigg(\frac{x}{\sqrt{u}}\bigg)
,\]
with an absolute implied constant.
\end{lemma}
\begin{proof}
It is easy to 
see that for positive $m$
one has $\vartheta_\Q(m)=1$ 
if and only if 
$m=2^t k^2 r$
for $t\in \Z_{\geq 0}$,
$k$ a positive integer all of whose primes are $3\md{4}$ and 
$r$ is such that $\varpi(r)=1$.
This shows that the sum over $m$ is 
\[
\sum_{\substack{(k,t)\in  \Z_{>0}  \times \Z_{\geq 0} 
\\
p\mid k \Rightarrow p\equiv 3 \md{4}}}
\sum_{\substack{r\in \Z_{>0} 
 \\
 r \leq x 2^{-t}k^{-2}
}}
\varpi(r)
\e(a_1 2^t k^2 r /q)
.\]
The contribution of the pairs $(k,t)$
with $2^t k^2>u$ is at most 
\[
\sum_{t\geq 0} \sum_{k>\sqrt{u 2^{-t}}} x2^{-t}k^{-2}
\ll
x \sum_{t\geq 0}  \frac{2^{-t}}{\sqrt{u 2^{-t}}}
\ll
\frac{x}{\sqrt{u}}
.\]
Noting that $\e(a_1 2^t k^2 r /q)$ as a function of $r$ is periodic modulo $q$ allows to partition all $r$ 
in congruences $\ell \in \Z/q\Z$, thus concluding the proof.
\end{proof}
The terms in the 
sum involving $\varpi$
in Lemma~\ref{lem:exponbrzer}
are in an arithmetic progression 
that is not necessarily primitive.
We next show that we can reduce the evaluation of
these sums 
to similar expressions 
where the summation is over an arithmetic progression that is primitive.  
The property~\eqref{eq:fulmul}
will be necessary for this.
\begin{lemma}
\label{lem:exponbrzergcds}
Let $t\in \Z_{\geq 0}$, 
$q\in \Z_{>0}$,
$\ell \in \Z\cap [0,q)$
and 
$k\in \Z_{>0}$ 
such that every prime divisor
of $k$ is $3\md{4}$.
For 
$y\in \R_{> 0}$
consider the sum 
\[
\sum_{\substack{r\in \Z_{>0} \cap [1,y] \\ 
2^t k^2 r \equiv \ell \md{q}
}}
\varpi(r)
.\]
The sum vanishes if 
$\gcd(2^tk^2,q)\nmid \ell$, and it otherwise equals  
\[
\varpi\Big(\frac{\gcd(\ell,q)}{\gcd(2^t k^2,q)}\Big)
\sum_{\substack{s\in \Z_{>0} \cap [1, y \gcd(2^tk^2,q)\gcd(\ell,q)^{-1}]
\\ 
\frac{2^tk^2}{\gcd(2^tk^2,q)}
s \equiv \frac{\ell}{\gcd(\ell,q)} \md{\frac{q}{\gcd(\ell,q)}}
}}
\varpi(s)
.\] 
\end{lemma}
\begin{proof}
If $\gcd(2^tk^2,q)\nmid \ell$ then 
the congruence 
$2^t k^2 r \equiv \ell \md{q}$ does not have a solution $r$,
in which case the sum over $r$ vanishes.
If it holds then we 
see that the congruence for $r$ can be written equivalently as 
\[
\frac{2^t k^2}{\gcd(2^t k^2,q)} r \equiv \frac{\ell}{\gcd(2^t k^2,q)} \md{\frac{q}{\gcd(2^t k^2,q)}}
.\]
Note that any solution $r$ of this must necessarily satisfy 
\[
\gcd
\Big(\frac{\ell}{\gcd(2^t k^2,q)}
,\frac{q}{\gcd(2^t k^2,q)}
\Big)
\Big|
\frac{2^t k^2}{\gcd(2^t k^2,q)} r 
\]
and the fact of 
\[\gcd
\Big(\frac{\gcd(\ell,q)}{\gcd(2^t k^2,q)}
,
\frac{2^t k^2}{\gcd(2^t k^2,q)}
\Big)
=1
\]
shows that $r$ must be divisible by $\gcd(\ell,q) \gcd(2^tk^2,q)^{-1}$.
Therefore there exists
an $s\in \Z_{>0}$ with
\[r=\frac{\gcd(\ell,q)}{\gcd(2^t k^2,q)}s\]
and substituting this into the sum over $r$ in our lemma 
concludes the proof because
\[
\varpi(r)=
\varpi\Big(\frac{\gcd(\ell,q)}{\gcd(2^t k^2,q)}\Big)
\varpi(s)
\]
holds due to 
the complete multiplicativity seen in
~\eqref{eq:fulmul}.
\end{proof} 
We  
are now in a position to apply~\cite[Th.14.7]{MR2647984},
which is
a result
on
the distribution of the function $\varpi$ along primitive
arithmetic progressions
and which we include as a proposition for the convenience of the reader.
We first introduce the following notation for $Q\in \Z_{>0}$,
\begin{equation}\label{def:qendio}
\dot{Q}:=
\prod_{p\equiv 1 \md{4}}
p^{v_p(Q)}
\
\text{ and }
\
\ddot{Q}:=
\prod_{p\equiv 3 \md{4}}
p^{v_p(Q)}
.\end{equation}
\begin{proposition}[\cite{MR2647984} Th.14.7]
\label{prop:halffdimen}
Assume that $Q$ is a positive integer that is a multiple of $4$,
that $a$ is an integer satisfying 
$\gcd(a,Q)=1,
a\equiv 1\md{4}
$
and let $z$ be any real number with $z\geq Q$.
Then 
\[
\sum_{\substack{r\in \Z_{>0}\cap [1,z]\\ r\equiv a \md{Q}}} \varpi(r)
=
2^{1/2}
\c{C}_0 \frac{\ddot{Q}
}{\phi(\ddot{Q})} 
\frac{z}{Q\sqrt{\log z}}
\bigg\{1+O\bigg(\Big(\frac{\log Q}{\log z}\Big)^{1/7}\bigg)
\bigg\}
,\]
where the implied constant is absolute.
\end{proposition}
\begin{remark}
This result was proved using 
the
semi-linear
Rosser--Iwaniec sieve.
We should remark that there is a typo
in the reference,
namely~\cite[Eq.(14.22)]{MR2647984}
should instead read
\[
V(D)=
\prod_{2<p<D}
\l(1-\frac{1}{p}\r)^{\frac{1}{2}}
\prod_{p<D}
\l(1-\frac{\chi(p)}{p}\r)^{-\frac{1}{2}}
\prod_{\substack{2<p<D\\p\equiv 3 \md{4}}}
\l(1-\frac{1}{p^2}\r)^{\frac{1}{2}}
,\]
and as a result,~\cite[Eq.(14.39)]{MR2647984} 
must be replaced by the asymptotic in Proposition~\ref{prop:halffdimen}. 
After fixing this typo,
one can show, 
as in the proof of~\cite[Eq.(14.24)]{MR2647984},
that for $D\geq 2$,
we have
\begin{equation}
\label
{eq:fixedtypo}
\prod_{\substack{p<D\\p\equiv 3 \md{4}}} 
\Big(1-\frac{1}{p}
\Big)
=
\frac{\sqrt{\pi}}{\sqrt{2 \e^\gamma}}
\c{C}_0 
\frac{1}{\sqrt{\log D}}
+O
\Big(
\frac{1}{(\log D)^{3/2}}
\Big)
.\end{equation}
There is a further typo in~\cite[Eq.(14.26)]{MR2647984},
namely, $c\sqrt{2}$ should be
replaced by 
$2^{1/2}
\c{C}_0 
/4.$
\end{remark}
We will
now proceed to the application of 
Proposition~\ref{prop:halffdimen}.
For $q\in \Z_{>0}, a_1\in \Z\cap [0,q)$
define 
\begin{equation}\label{def:snowdurham}
\mathfrak{F}(a_1,q)
\!
:=
\!
\hspace{-0,5cm}
\sum_{\substack{(k,t)\in  \Z_{>0}  \times \Z_{\geq 0} \\
p\mid k \Rightarrow p\equiv 3 \md{4}}}
\hspace{-0,5cm}
\frac{\gcd(2^tk^2,q)}{2^{t}k^{2}}
\hspace{-0,6cm}
\sum_{\substack{
\ell \in \Z\cap [0,q)
\\
\gcd(2^tk^2,q)\mid \ell
,\eqref{eq:albinioni}}}
\hspace{-0,4cm}
\frac{\varpi\big(\frac{\gcd(\ell,q)}{\gcd(2^t k^2,q)}\big)
\e\big(\frac{a_1 \ell}{q}\big)}{\gcd(\ell,q)\mathrm{lcm}\big(4,\frac{q}{\gcd(\ell,q)}\big)}
\prod_{\substack{ 
p\equiv 3 \md{4}
\\
v_p(q)>v_p(\ell)
}}  
\Big(1-\frac{1}{p}\Big)^{-1}
\!,\end{equation}
where  $\ell$ in the summation satisfies
\begin{equation}\label{eq:albinioni}
\frac{2^tk^2}{\gcd(2^tk^2,q)}\equiv
\frac{\ell}{\gcd(\ell,q)} 
\md{\gcd\Big(4,\frac{q}{\gcd(\ell,q)}\Big)}
.\end{equation}
The result of the following lemma aims to separate out the dependence 
on $x$  
from the apparent 
pandemonium
 that is hidden in $\mathfrak{F}(a_1,q)$.
\begin{lemma}
\label{lem:halffdimenrtyui}
For 
$x\in \R_{\geq 1},
A\in \R_{> 0},
q\in \Z_{>0},
a_1\in \Z\cap [0,q)$ 
with $q\leq (\log x)^{A}$
we have 
\[
\sum_{\substack{m\in \Z \cap [1,x]
\\
x_0^2+x_1^2=m x_2^2
\text{ has a }\Q\text{-point}
}} 
\hspace{-0,5cm}
\e^{2\pi i a_1
\frac{m}{q}}
=
2^{1/2}
\c{C}_0
\mathfrak{F}(a_1,q)
\frac{x}{(\log x)^{1/2}}
+O\bigg(\frac{q^3 x}{(\log x)^{1/2+1/7}}\bigg)
,\]
where 
the implied constant depends at most on $A$.
\end{lemma}
\begin{proof}
Combining 
Lemma~\ref{lem:exponbrzer} with $u=(\log x)^{100}$ 
and 
Lemma~\ref{lem:exponbrzergcds}
shows that, up to an error term which is $\ll x (\log x)^{-50}$, 
the sum over $m$ in our lemma equals 
\[ 
\sum_{\substack{(k,t)\in  \Z_{>0}  \times \Z_{\geq 0} \\
2^t k^2 \leq (\log x)^{100} 
\\
p\mid k \Rightarrow p\equiv 3 \md{4}}}
\sum_{\substack{
\ell \in \Z\cap [0,q)
\\
\gcd(2^tk^2,q)\mid \ell
}}
\varpi\Big(\frac{\gcd(\ell,q)}{\gcd(2^t k^2,q)}\Big)
\e(a_1 \ell/q)
\sum_{\substack{s\in \Z_{>0} \cap [1,x 2^{-t}k^{-2} \gcd(2^tk^2,q)\gcd(\ell,q)^{-1}] \\ 
\frac{2^tk^2}{\gcd(2^tk^2,q)}
s \equiv \frac{\ell}{\gcd(\ell,q)} \md{\frac{q}{\gcd(\ell,q)}}
}}
\varpi(s)
.\]
We note that $\varpi(s)$ vanishes unless $s\equiv 1 \md{4}$.
This means that we can add the condition $s\equiv 1 \md{4}$ in the last sum over $s$,
thus resulting with the double congruence 
\[
s\equiv 1 \md{4},
\frac{2^tk^2}{\gcd(2^tk^2,q)}
s \equiv \frac{\ell}{\gcd(\ell,q)} \md{\frac{q}{\gcd(\ell,q)}}
.\]
By the 
Chinese remainder theorem this has a solution 
if and only if~\eqref{eq:albinioni} is satisfied.
Assuming that this happens, 
the solution is unique modulo $$Q:=\mathrm{lcm}\Big(4,\frac{q}{\gcd(\ell,q)}\Big),$$ hence by 
Proposition~\ref{prop:halffdimen}
we get that 
the sum over $m$ in our lemma equals 
\begin{align*}
\text{MT}
:&=
2^{1/2}
\c{C}_0 
\sum_{\substack{(k,t)\in  \Z_{>0}  \times \Z_{\geq 0} \\
2^t k^2 \leq (\log x)^{100} 
\\
p\mid k \Rightarrow p\equiv 3 \md{4}}}
\sum_{\substack{
\ell \in \Z\cap [0,q),
\eqref{eq:albinioni}
\\
\gcd(2^tk^2,q)\mid \ell
}}
\varpi\Big(\frac{\gcd(\ell,q)}{\gcd(2^t k^2,q)}\Big)
\e(a_1 \ell/q) \times 
\\ & 
\times 
\frac{\ddot{Q}}{\phi(\ddot{Q})} \frac{1}{\mathrm{lcm}(4,q/\gcd(\ell,q))}
\frac{x 2^{-t}k^{-2} \gcd(2^tk^2,q)\gcd(\ell,q)^{-1}}{\sqrt{\log (x 2^{-t}k^{-2} \gcd(2^tk^2,q)\gcd(\ell,q)^{-1})}}
\end{align*}
up to an error term which is 
\begin{equation}\label{eq:uptoer}
\ll 
\frac{x}{(\log x)^{50}}
+
\hspace{-0,5cm}
\sum_{\substack{(k,t)\in  \Z_{>0}  \times \Z_{\geq 0} \\
2^t k^2 \leq (\log x)^{100} 
\\
p\mid k \Rightarrow p\equiv 3 \md{4}}}
\hspace{-0,1cm}
\sum_{\substack{
\ell \in \Z\cap [0,q),
\eqref{eq:albinioni}
\\
\gcd(2^tk^2,q)\mid \ell
}} 
(\log \log \ddot{Q})
\frac{x 2^{-t}k^{-2} \gcd(2^tk^2,q)
}{\gcd(\ell,q)\sqrt{\log x}}
\Big(\frac{\log Q}{\log x}\Big)^{1/7} 
\end{equation}
owing to~\eqref{def:phiqridten}, which gives  
$\ddot{Q}/\phi(\ddot{Q})\ll \log \log \ddot{Q}\leq
\log \log Q$.
Note that we have made use of 
\begin{equation}\label{eq:bwvthat}
\log (x 2^{-t}k^{-2} \gcd(2^tk^2,q)\gcd(\ell,q)^{-1})
=\log x+O_{A}(\log \log x)
,\end{equation}
which
follows from
\[
\frac{x}{(\log x)^{100+A}} 
\leq 
\frac{x}{2^tk^2q}
\leq 
x 2^{-t}k^{-2} \gcd(2^tk^2,q)\gcd(\ell,q)^{-1}
\leq x q \leq 
x(\log x)^{A}
.\] 
The bound $\ddot{Q}\leq Q\leq 4q$ shows that the sum over $t,k$
in~\eqref{eq:uptoer} is 
\begin{align*}
\ll
&(\log \log q)
(\log q)^{1/7}
\frac{x}{(\log x)^{1/2+1/7}}
\sum_{\substack{(k,t)\in  \Z_{>0}  \times \Z_{\geq 0} }}
\hspace{-0,1cm}
\sum_{\substack{
\ell \in \Z\cap [0,q) 
}} 
2^{-t}k^{-2} \gcd(2^tk^2,q)
\\
\ll
&(\log \log q)
(\log q)^{1/7}
\frac{x}{(\log x)^{1/2+1/7}}
q^2
\sum_{\substack{(k,t)\in  \Z_{>0}  \times \Z_{\geq 0}}}
2^{-t}k^{-2} 
\ll
q^3
\frac{x}{(\log x)^{1/2+1/7}}
,\end{align*}
which is satisfactory.
To conclude the proof, it remains to show that the quantity $\text{MT}$
gives rise to the main term as stated in our lemma.
By~\eqref{eq:bwvthat}
we see that 
\[\frac{1}{\sqrt{\log (x 2^{-t}k^{-2} \gcd(2^tk^2,q)\gcd(\ell,q)^{-1})}}
=\frac{1}{\sqrt{\log x}}
+O\bigg(
\frac{\log \log x}{(\log x)^{3/2}}
\bigg)
,\]
hence $\text{MT}=\text{M}'+\text{R}$, where $\text{M}'$ is defined by   
\[
\frac{x2^{1/2}\c{C}_0}{(\log x)^{1/2}}
\sum_{\substack{(k,t)\in  \Z_{>0}  \times \Z_{\geq 0} \\
2^t k^2 \leq (\log x)^{100} 
\\
p\mid k \Rightarrow p\equiv 3 \md{4}}}
\frac{\gcd(2^tk^2,q)}{2^{t}k^{2}}
\sum_{\substack{
\ell \in \Z\cap [0,q)
\\
\gcd(2^tk^2,q)\mid \ell
,
\eqref{eq:albinioni}
}}
\frac{\varpi(\gcd(\ell,q)/\gcd(2^t k^2,q))
\e(a_1 \ell/q)\ddot{Q}}{\gcd(\ell,q)
\mathrm{lcm}(4,q/\gcd(\ell,q))
\phi(\ddot{Q})}
\]
and $\text{R}$ is a quantity that 
satisfies 
\[\text{R}\ll
\sum_{\substack{(k,t)\in  \Z_{>0}  \times \Z_{\geq 0} }}
\sum_{\substack{
\ell \in \Z\cap [0,q) 
}} 
\frac{\ddot{Q}}{\phi(\ddot{Q})} 
\frac{x 2^{-t}k^{-2} \gcd(2^tk^2,q)}{(\log \log x)^{-1}(\log x)^{3/2}}
\ll
q^3
\frac{x\log \log x}{(\log x)^{3/2}}
.\]
To complete the summation over $t,k$ in $\text{M}'$
we use 
the bound 
\[ 
\sum_{\substack{(k,t)\in  \Z_{>0}  \times \Z_{\geq 0} \\
2^t k^2 >
(\log x)^{100} }}
\hspace{-0,5cm}
\frac{\gcd(2^tk^2,q)}{2^{t}k^{2}}
\hspace{-0,5cm}
\sum_{\substack{
\ell \in \Z\cap [0,q),
\eqref{eq:albinioni}
\\
\gcd(2^tk^2,q)\mid \ell
}}
\frac{\ddot{Q}}{\phi(\ddot{Q})} 
\ll q^3 \sum_{\substack{(k,t)\in  \N  \times \Z_{\geq 0} \\
2^t k^2 >
(\log x)^{100} }} \frac{1}{2^t k^2} \ll \frac{q^3}{(\log x)^{{50}}}
,\]
while the observation 
\[\frac{\ddot{Q}}{\phi(\ddot{Q})} =
\prod_{\substack{ 
p\equiv 3 \md{4}
\\
p\mid q(\gcd(\ell,q))^{-1}
}} \Big(1-\frac{1}{p}\Big)^{-1}
=
\prod_{\substack{ 
p\equiv 3 \md{4}
\\
v_p(q)>v_p(\ell)
}} \Big(1-\frac{1}{p}\Big)^{-1}
\]
allows to remove $\ddot{Q}$ from $\text{M}'$.
\end{proof}
We note that one immediate corollary of the last lemma 
is
the bound
\begin{equation}\label{eq:imedcor}
\mathfrak{F}(a_1,q)\ll 1,\end{equation}
with an absolute implied constant.
Indeed, this can be shown by taking $A=1/100$ in 
Lemma~\ref{lem:halffdimenrtyui},
dividing throughout by $x/\sqrt{\log x}$ in the asymptotic it provides and 
alluding to~\eqref{eq:landaurzero} 
to obtain   
\begin{align*}
2^{1/2}
\c{C}_0
\mathfrak{F}(a_1,q)
&\ll
\frac{(\log x)^{1/2}}{x}\Big|\sum_{1\leq m \leq x} \vartheta_\Q(m) \e(a_1 m/q) \Big|
+\frac{q^3}{(\log x)^{1/7}} 
\\ &\ll 1 +\frac{\ (\log x)^{3/100}}{(\log x)^{1/7}} 
.\end{align*}

As announced at the beginning of this section, studying $E_\Q\Big( \frac{a_1}{q}+\beta_1\Big)$ is first done in the case $\beta_1=0$ as above. The following lemma shows that this is sufficient, up to introducing an extra factor.

\begin{lemma}\label{lem:easypiezy}
For 
$\Gamma_1 \in \R,
A\in \R_{>0},
q\in \Z_{>0},
a_1\in \Z\cap [0,q)$ 
with $q\leq (\log P)^{A}$
we have 
\begin{align*}
E_\Q\Bigg(\frac{a_1}{q}+\frac{\Gamma_1}{P^{d}}\Bigg)
&=
2^{1/2}
\c{C}_0
\mathfrak{F}(a_1,q)
\bigg(\int_{2}^{\max\{f_1([-1,1]^n)\}P^d}
\hspace{-0,1cm}
\frac{\e(\Gamma_1 P^{-d} t)}{\sqrt{\log t}}
\mathrm{d}t\bigg)
+
O_A\bigg(
\frac{q^3
(1+|\Gamma_1|)P^d}{(\log P)^{1/2+1/7}}
\bigg)
,\end{align*} 
with an implied constant depending at most on $A$.
\end{lemma}
\begin{proof}
To ease the notation we temporarily put
$
c:=2^{1/2}
\c{C}_0 
\mathfrak{F}(a_1,q)
$. Fix
$\beta \in \R$.
By
partial summation 
 $
\sum_{m\leq x}\vartheta_\Q(m)
\e(m(\beta+a_1/q))
$
equals 
\[
\Big(\sum_{m\leq x}\vartheta_\Q(m)\e(a_1m/q)\Big)
\e(x\beta)
-\int_{0}^x
\e(\beta t)'
\Big(\sum_{m\leq t}\vartheta_\Q(m)\e(a_1m/q)
\Big)
\mathrm{d}t
.\]
If  
$q\leq (\log x)^A$ then Lemma~\ref{lem:halffdimenrtyui}
shows that this
equals
\[
c\Big(
\Big(
\frac{x}{\sqrt{\log x}}
\e(x\beta)
-\int_{2}^x
\frac{t}{\sqrt{\log t}}
 \e(\beta t)'
\mathrm{d}t
\Big)
+O\Big(\frac{q^3x(1+|\beta| x) }{(\log x)^{1/2+1/7}}\Big)
,\]
with an implied constant depending at most on $A$.
Using partial integration
this is plainly 
 \[
c\Big(
\int_{2}^x
\Big(\frac{t}{\sqrt{\log t}} \Big)'
\e(\beta t)
\mathrm{d}t
\Big)
+O\Big(\frac{q^3(1+|\beta| x)x}{(\log x)^{1/2+1/7}}\Big)
,\]
and using $(t (\log t)^{-1/2})'=(\log t)^{-1/2}-2^{-1}(\log t)^{-3/2}$ 
shows that the last integral can be evaluated as 
$\int_{2}^x
\e(\beta t) (\log t)^{-1/2}
\mathrm{d}t+
O(x (\log x)^{-3/2})$.
Invoking  the bound $c\ll 1$ (that is implied by~\eqref{eq:imedcor})
we obtain   \[\sum_{m\leq x}\vartheta_\Q(m)
\e(m(\beta+
a_1/q))=
c\Big(
\int_{2}^x
\frac{\e(\beta t)}{\sqrt{\log t}}
\mathrm{d}t
\Big)
+O\bigg( \frac{q^3(1+|\beta| x)x}{(\log x)^{1/2+1/7}}\bigg)
.\]
Using this 
for
$x=
\max\{f_1([-1,1]^n)\} P^d  
  $
and  putting $\beta=\Gamma_1 P^{-d}$
concludes the proof.
\end{proof}

\section{Proof of the asymptotic}
\label{s:prfasform}
We are
ready to prove
the asymptotic in Theorem~\ref{thm:main1}.
We shall do so with different leading constants than those given in Theorem~\ref{thm:main1}; showing equality of the constants is delayed until~\S\ref{s:archdens}. 
\subsection{Restricting the range in the major arcs}
\label{sectruncmajor}
The first reasonable
step for the proof of the asymptotics 
would be to inject
Lemma~\ref{lem:easypiezy}
into 
Lemma~\ref{lem:nextinbirch}.
However, 
this would give
poor results 
because the 
error term in Lemma~\ref{lem:easypiezy}
is only
powerful when 
$\Gamma_1$ is close to zero and $q$ 
is small in comparison to $P$.
For this reason 
we restrict  
the sum over $q$
and the integration over $\boldsymbol{\beta}$ in Lemma~\ref{lem:nextinbirch}.
For its proof we shall need 
certain bounds.
First,
by~\eqref{eq:landaurzero},
one has  
\begin{equation}\label{def:exponentrzerfgrt}
E_\Q(\alpha_1)
\ll P^d (\log P)^{-1/2}
.\end{equation} 
Next,
letting
$K:=
(n-\sigma(f_1,f_2))
2^{-d+1}$, we use
~\cite[Lem.5.2, Lem.5.4]{birch}  
to obtain the following 
bounds for every $\epsilon>0$,
$\boldsymbol{\Gamma} \in \R^2$ 
and $\b{a}\in \Z^2,q\in \N$ satisfying $\gcd(a_1,a_2,q)=1$:
\[I(\boldsymbol{\Gamma})
\ll_\epsilon
\min\{1,|\boldsymbol{\Gamma}|^{-K/(2(d-1))+\epsilon}\}
\text{ and }
S_{\b{a},q}\ll_\epsilon
q^{n-K/(2(d-1))+\epsilon}
.\]
By
our assumption~\eqref{ass:Birch}, we have 
\begin{equation}\label{eq:inthortobagyi}
I( \boldsymbol{\Gamma})
\ll
\min\{1,|\boldsymbol{\Gamma}|^{-5/2}\}
,\end{equation}
and, furthermore, 
that 
for all  
$0<\lambda < 2^{-d} (d-1)^{-1}$ we have
\begin{equation}\label
{eq:sumhortobagyi}
S_{\b{a},q}\ll_\lambda 
q^{n-3-\lambda}
.\end{equation}
\begin{lemma}
\label{lem:finredbl}
Keep the assumptions of Lemma~\ref{lem:asinbirch}
and let $Q_1,Q_2\in \R_{\geq 1}$
with $Q_1,Q_2\leq P^\eta$. 
Then for any  $\lambda$ satisfying 
\begin{equation}\label{def:lam}
0<\lambda<
\min\Big\{1,
\frac{1}{2}
\Big(\frac{n-\sigma(f_1,f_2)}{2^d (d-1)}-3\Big)
\Big\}
,\end{equation}
we have 
\begin{align*}
&\sum_{q\leq P^{\eta}} q^{-n}
\sum_{\substack{
\b{a}\in (\Z\cap [0,q))^2
\\
\gcd(a_1,a_2,q)=1
}}S_{\b{a},q}
\int_{|\boldsymbol{\beta}|\leq P^{-d+\eta}}
P^d I( P^d\boldsymbol{\beta})
\overline{E_\Q(\beta_1+a_1/q)}
\mathrm{d}\boldsymbol{\beta}
\\
=&\sum_{q\leq Q_1} q^{-n}
\sum_{\substack{
\b{a}\in (\Z\cap [0,q))^2
\\
\gcd(a_1,a_2,q)=1
}}S_{\b{a},q}
\int_{|\boldsymbol{\Gamma}|\leq Q_2}
\frac{I( \boldsymbol{\Gamma})}{P^d}
\overline{E_\Q(\Gamma_1P^{-d}+a_1/q)}
\mathrm{d}\boldsymbol{\Gamma}
\\
+&O_{\delta,\lambda,\theta_0}
\big((\log P)^{-1/2} \min\big\{Q_1^{-\lambda},Q_2^{-1/2}\big\}
\big)
.\end{align*} 
\end{lemma}
\begin{proof}
Using the change of variables $P^d\boldsymbol{\beta}=\boldsymbol{\Gamma}$
we obtain
\[\int_{
P^{-d}Q_2<
|\boldsymbol{\beta}|\leq P^{-d+\eta}}
\hspace{-0,2cm}
P^d I( P^d\boldsymbol{\beta})
\overline{E_\Q(\beta_1+a_1/q)}
\mathrm{d}\boldsymbol{\beta}
=P^{-d}
\int_{Q_2<|\boldsymbol{\Gamma}|\leq P^{\eta}}
\hspace{-0,2cm}
I( \boldsymbol{\Gamma})
\overline{E_\Q(\Gamma_1P^{-d}+a_1/q)}
\mathrm{d}\boldsymbol{\Gamma}
\]
and
combining~\eqref{def:exponentrzerfgrt}
with~\eqref{eq:inthortobagyi}
shows
that  
\begin{equation}
\label{eq:bond01}
\int_{P^{-d}Q_2<|\boldsymbol{\beta}|\leq P^{-d+\eta}}
P^d I(P^d\boldsymbol{\beta})
\overline{E_\Q(\beta_1+a_1/q)}
\mathrm{d}\boldsymbol{\beta}
\ll 
\frac{1}{\sqrt{Q_2 \log P}}
.\end{equation}
This shows that the sum over $q\leq P^\eta$ in the statement of our lemma  
equals 
\[
\sum_{q\leq P^{\eta}} q^{-n}
\sum_{\substack{
\b{a}\in (\Z\cap [0,q))^2
\\
\gcd(a_1,a_2,q)=1
}}S_{\b{a},q}
\int_{|\boldsymbol{\Gamma}|\leq Q_2
}
\frac{I(\boldsymbol{\Gamma})}{P^d}
\overline{E_\Q(\Gamma_1P^{-d}+a_1/q)}
\mathrm{d}\boldsymbol{\Gamma}
,\]
up to a term 
that is 
\[
\ll
\frac{1}{
\sqrt{Q_2 \log P}
}
\sum_{q\leq P^{\eta}} 
\sum_{\substack{
\b{a}\in (\Z\cap [0,q))^2
\\
\gcd(a_1,a_2,q)=1
}}\frac{|S_{\b{a},q}|}{q^n}
\ll
\frac{\sum_{q\leq P^{\eta}}  q^{-1-\lambda} }{
\sqrt{Q_2 \log P}
}
\ll
\frac{1}{
\sqrt{Q_2 \log P}
}
,\]
where~\eqref{eq:sumhortobagyi} has been utilised.
Note that the bound 
$\int_{\R^2} 
|I(\boldsymbol{\Gamma}) |
\mathrm{d}\boldsymbol{\Gamma}
<\infty$ 
is a consequence of~\eqref{eq:inthortobagyi}.
Using this   with~\eqref{def:exponentrzerfgrt}
shows that
\[\sum_{q>Q_1} \sum_{\substack{
\b{a}\in (\Z\cap [0,q))^2
\\
\gcd(a_1,a_2,q)=1
}}
\frac{S_{\b{a},q}}{q^n}
\int_{|\boldsymbol{\beta}|\leq P^{-d}Q_2}
\hspace{-0,1cm}
P^{d}  I(P^d\boldsymbol{\beta}) 
\overline{E_\Q(\beta_1+a_1/q)}
\mathrm{d}\boldsymbol{\beta}
\ll \frac{\sum_{q>Q_1} q^{-1-\lambda}}{\sqrt{\log P}}
\ll
\frac{
Q_1^{-\lambda}
}{\sqrt{\log P}}
,\] 
where we have alluded to~\eqref{eq:sumhortobagyi}.
This concludes
the proof of the lemma.
\end{proof}

\begin{lemma}
\label{lem:finredbllach}
Keep the assumptions of Lemma~\ref{lem:asinbirch},
fix any two positive
$A_1,A_2$
and 
let 
\begin{equation}\label{def:lambdazero}
\lambda_0:=\frac{1}{2}
\min\Bigg\{1,
\frac{1}{2}
\Bigg(\frac{n-\sigma(f_1,f_2)}{2^d (d-1)}-3\Bigg)
\Bigg\}.\end{equation}
Then 
for all sufficiently large $P$
the quantity 
$\Theta_\Q(P)
P^{-n+d}$
equals 
\[
\sum_{q\leq (\log P)^{A_1}}  
\hspace{-0,1cm}
\sum_{\substack{
\b{a}\in (\Z\cap [0,q))^2
\\
\gcd(a_1,a_2,q)=1
}}
\hspace{-0,1cm}
\frac{S_{\b{a},q}}{q^n}
\int_{|\boldsymbol{\Gamma}|\leq (\log P)^{A_2}}
\hspace{-0,2cm}
\frac{I(\boldsymbol{\Gamma})}{P^d}
\overline{E_\Q\bigg(\frac{a_1}{q}+\frac{\Gamma_1}{P^{d}}\bigg)}
\mathrm{d}\boldsymbol{\Gamma}
+O_{A_1,A_2}\big((\log P)^{-1/2-\min\{A_1 \lambda_0,A_2/2\}}\big)
.\] 
\end{lemma}
\begin{proof}
The proof follows immediately by 
using Lemma~\ref{lem:finredbl}
with
$Q_i=(\log P)^{A_i}$ 
and 
Lemma~\ref{lem:nextinbirch}
with some fixed $\eta$ and $\theta_0$ satisfying~\cite[pg.251,Eq.(10)-(11)]{birch}
and
$\eta<1/7$.
\end{proof}
\subsection{Injecting the sieve estimates into the restricted major arcs}
\label{s:profastic}
We are now in position
to inject Lemma~\ref{lem:easypiezy}
into Lemma~\ref{lem:finredbllach}.
It will be convenient to start by 
studying the archimedean density. 
Recall~\eqref{def:singintegral}
and
define for $P>3/\min\{f_1([-1,1]^n)\}$,
\begin{equation}\label{def:singint}
\mathfrak{J}_\upphi(P)
:=
\int_{\boldsymbol{\Gamma} \in \R^2}
\frac{I(\boldsymbol{\Gamma})}{P^d}
\bigg(\int_{3}^{\max\{f_1([-1,1]^n)\}P^d}
\frac{\e(-\Gamma_1 P^{-d} t)}{\sqrt{\log t}}
\mathrm{d}t\bigg)
\mathrm{d}\boldsymbol{\Gamma} 
.\end{equation}
The assumptions of Theorem~\ref{thm:main1} ensure that 
the integral converges absolutely, 
since by~\eqref{eq:inthortobagyi}
we have
\[
\int_{\boldsymbol{\Gamma} \in \R^2}
\frac{|I(\boldsymbol{\Gamma})|}{P^d}
\int_{2}^{\max\{f_1([-1,1]^n)\}P^d}
\frac{\mathrm{d}t}{\sqrt{\log t}}
\mathrm{d}\boldsymbol{\Gamma} 
\ll
\int_{\boldsymbol{\Gamma} \in \R^2}
\frac{\min\{1,|\boldsymbol{\Gamma}|^{-5/2}\}}{P^d}
\frac{P^d}{\sqrt{\log P}} 
\ll \frac{1}{\sqrt{\log P}}
.\]
\begin{lemma}\label{lem:archdensyt}
Under the assumptions of Theorem~\ref{thm:main1} we have  
\[
\mathfrak{J}_\upphi(P)
=
\frac{1}{\sqrt{\log (P^d)}} 
\int_{\boldsymbol{\Gamma} \in \R^2}
I(\boldsymbol{\Gamma})
\bigg(\int_{0}^{\max\{f_1([-1,1]^n)\}} 
\hspace{-0,2cm}
\e(-\Gamma_1 \mu) 
\mathrm{d}\mu\bigg)
\mathrm{d}\boldsymbol{\Gamma}
+O((\log P)^{-3/2})
.\]
\end{lemma}
\begin{proof}
Observe that 
the
change of variables 
$\mu=P^{-d}t$  in~\eqref{def:singint}
shows that 
\[
\mathfrak{J}_\upphi(P)=
\int_{\boldsymbol{\Gamma} \in \R^2}
I(\boldsymbol{\Gamma})
\bigg(\int_{3P^{-d}}^{\max\{f_1([-1,1]^n)\}}
\frac{\e(-\Gamma_1 \mu)}{\sqrt{\log (\mu P^d)}}
\mathrm{d}\mu\bigg)
\mathrm{d}\boldsymbol{\Gamma}
.\]
It is easy to verify that
$
(1+x)^{-1/2}=
1+O(x)$ for $|x|< 1$, hence for $\mu$ and $P$ in the range 
$0<\mu<P^d$ we have 
\[
\big(\log (\mu P^d)\big)^{-1/2}
=
\big(\log (P^d)\big)^{-1/2}
\bigg(1+\frac{\log \mu}{\log (P^d)}\bigg)^{-1/2}
=
\big(\log (P^d)\big)^{-1/2}
+O\bigg(\frac{\log \mu}{(\log P)^{3/2}}\bigg).\]
Using this for $0<\mu\leq \max\{f_1([-1,1]^n)\}$, 
we infer 
the following estimate 
for all sufficiently large $P$,
\[
\mathfrak{J}_\upphi(P)
-
\frac{1}{\sqrt{\log (P^d)}} 
\int_{\boldsymbol{\Gamma} \in \R^2}
I(\boldsymbol{\Gamma})
\int_{3 P^{-d}}^{\max\{f_1([-1,1]^n)\}}
\e(-\Gamma_1 \mu) 
\mathrm{d}\mu 
\mathrm{d}\boldsymbol{\Gamma}
\ll
\frac{1}{(\log P)^{3/2}} 
\int_{\boldsymbol{\Gamma} \in \R^2}
 | I(\boldsymbol{\Gamma}) | 
\mathrm{d}\boldsymbol{\Gamma}
,\]
which is $\ll (\log P)^{-3/2}$
due to~\eqref{eq:inthortobagyi}.
\end{proof}
Define 
\begin{equation}\label{def:mintokratas}
\mathfrak{J}
:=\int_{\Gamma \in \R}
\int_{\substack{
\{\b{t} \in  [-1,1]^n:
x_0^2+x_1^2=f_1(\b{t})x_2^2 \text{ has an } \R\text{-point}
\}
}} \e(\Gamma f_2(\b{t})) 
\mathrm{d}\b{t}
\mathrm{d}\Gamma 
\end{equation}
and note that the integral 
converges absolutely
owing to the smoothness
of   $f_1$ and $f_2$,~\eqref{ass:Birch} and~\cite[Lem.5.2]{birch}  
with $R=1$. 
The arguments in~\cite[\S 6]{birch} 
show that 
if there is 
 a non-singular real
point of $f_2=0$
contained in 
the set $\{\b{t}\in [-1,1]^n:f_1(\b{t})\geq 0\}$
  then $\mathfrak{J}
>0$.
In the situation of Theorem~\ref{thm:main1}
this condition holds, because   its assumptions 
include that 
$B(\Q)  \neq \emptyset$ 
and that $f_2$ is non-singular.  
\begin{lemma}
\label{lem:mozart_k427} 
Under the assumptions of Theorem~\ref{thm:main1}
we have  
 \[\int_{\boldsymbol{\Gamma} \in \R^2}
I(\boldsymbol{\Gamma})
\bigg(\int_{0}^{\max\{f_1([-1,1]^n)\}}
\e(-\Gamma_1 \mu) 
\mathrm{d}\mu\bigg)
\mathrm{d}\boldsymbol{\Gamma}
=
\mathfrak{J}
.\]
\end{lemma}
\begin{proof}
Define for $m\in \N$
the function
$\phi_m:\R\to\R$ through 
$\phi_m(x):=\pi^{-1/2}m \exp(-m^2 x^2)$.
First one may show
\[
\lim_{m\to+\infty} 
\int_{\boldsymbol{\Gamma} \in \R^2}
\frac{I(\boldsymbol{\Gamma})}
{\e^{\pi^2\Gamma_1^2 m^{-2}}}
\int_{0}^{\max\{f_1([-1,1]^n)\}}
\hspace{-0,2cm}
\e(-\Gamma_1 \mu) 
\mathrm{d}\mu
\mathrm{d}\boldsymbol{\Gamma}
=
\int_{\boldsymbol{\Gamma} \in \R^2}
I(\boldsymbol{\Gamma})
\int_{0}^{\max\{f_1([-1,1]^n)\}}
\hspace{-0,2cm}
\e(-\Gamma_1 \mu) 
\mathrm{d}\mu 
\mathrm{d}\boldsymbol{\Gamma}
,\]
for example by considering the difference between the right-hand side of this equality and each individual member of the limit on the left-hand side. Then one shows that this difference is $o(1)$ independently of $m$, by splitting the integral over $\Gamma_1$ up into the ranges $0<|\Gamma_1|<\log m$ and $\log m<|\Gamma_1|$ and showing that the two resulting integrals are both $o(1)$. One will need~\eqref{eq:inthortobagyi} for this.

Recalling~\eqref{def:singintegral}
and using the following formula with $x=
f_1(\b{t})-\mu$,
\[
\phi_m(x)=\int_\R
\e^{-\pi^2\Gamma_1^2 m^{-2}}
\e(x\Gamma_1) 
\mathrm{d}
\Gamma_1
,\]
that can be established by Fourier's inversion formula,
allows us to
rewrite 
the integral inside the limit as 
\[
\int_{\substack{
\b{t} \in [-1,1]^n:
f_1(\b{t}) \neq 0 \\
f_1(\b{t}) \neq \max\{f_1([-1,1]^n)\}
}}
\Bigg(\int_{0}^{} \phi_m(f_1(\b{t})-\mu)\mathrm{d}\mu\Bigg) 
\Bigg(\int_{\Gamma_2\in \R} \e(\Gamma_2 f_2(\b{t})) 
\mathrm{d}\Gamma_2\Bigg)
\mathrm{d}\b{t}
.\]
Note that we used~\eqref{def:singintegral}
with  $[-1,1]^n$
replaced by the range of integration for $\b{t}$ in the expression
above;
this is clearly allowable as it only removes a set of measure zero 
from the integration in~\eqref{def:singintegral}.
It is now easy to see that 
the limit 
\[
\lim_{m\to+\infty}
\int_{c_1}^{c_2}
\phi_m(\mu) 
\mathrm{d}\mu
\]
equals $1$ if $c_1<0<c_2$
and that it 
vanishes when $c_1>0$.
This proves that 
if $\b{t}\in [-1,1]^n$ 
satisfies $f_1(\b{t})>0$, then the limit   
\[
\lim_{m\to+\infty}
\int_{0}^{\max\{f_1([-1,1]^n)\}} 
\phi_m(f_1(\b{t})-\mu)\mathrm{d}\mu 
\]
equals $1$, while, if $f_1(\b{t})<0$
then the limit vanishes.
The dominated convergence theorem then 
gives
 \[
\int_{\boldsymbol{\Gamma} \in \R^2}
I(\boldsymbol{\Gamma})
\int_{0}^{\max\{f_1([-1,1]^n)\}}
\e(-\Gamma_1 \mu) 
\mathrm{d}\mu 
\mathrm{d}\boldsymbol{\Gamma}
=
\int_{\substack{
\b{t} \in [-1,1]^n:
f_1(\b{t}) \neq 0 \\
f_1(\b{t}) \neq \max\{f_1([-1,1]^n)\}
}}
\Bigg(\int_{\Gamma_2\in \R} \e(\Gamma_2 f_2(\b{t})) 
\mathrm{d}\Gamma_2\Bigg)
\mathrm{d}\b{t}
,\]
which concludes the proof.
\end{proof}

Having dealt with the integral part of Lemma~\ref{lem:nextinbirch}, we now turn our attention to the summation.
Recall the
definition of 
$S_{\b{a},q}$ and $\mathfrak{F}(a_1,q)$
respectively in~\eqref{def:singserr}
and~\eqref{def:snowdurham}
and 
let  
\begin{equation}\label{def:singser}
\mathbb{L}_\upphi
:=\sum_{q\in \mathbb{N}}  
q^{-n}
\sum_{\substack{
\b{a}\in (\Z\cap [0,q))^2
\\
\gcd(a_1,a_2,q)=1
}}
S_{\b{a},q}
\overline{\mathfrak{F}(a_1,q)}
.\end{equation} 
Under
the assumptions of Theorem~\ref{thm:main1}
the sum
$\mathbb{L}_\upphi$
converges absolutely, 
since by~\eqref{eq:imedcor} and~\eqref{eq:sumhortobagyi}
we have for all $x>1$,
\begin{equation}
\label{eq:toutreloumouerik}
\sum_{\substack{q\in \mathbb{N}\\ q>x}}
q^{-n}
\sum_{\substack{
\b{a}\in (\Z\cap [0,q))^2
\\
\gcd(a_1,a_2,q)=1
}}
\big|S_{\b{a},q} 
\overline{\mathfrak{F}(a_1,q)}\big|
\ll
\sum_{\substack{q\in \mathbb{N}\\ q>x}  }
q^{-n}
\sum_{\substack{
\b{a}\in (\Z\cap [0,q))^2
\\
\gcd(a_1,a_2,q)=1
}} q^{n-3-\lambda_0}
\leq 
\sum_{\substack{q\in \mathbb{N}\\ q>x}  } q^{-1-\lambda_0} \ll x^{-\lambda_0}
.\end{equation}
\begin{lemma}
\label{lem:pliades}
Under the assumptions of Theorem~\ref{thm:main1} we have for all $P\geq 2$,
\[
\Theta_\Q(P)
=
\c{C}_0 
\mathfrak{J}
\frac{
\mathbb{L}_\upphi
\sqrt{2}
}{d^{1/2}}
\frac{P^{n-d}}{(\log P)^{1/2}}
+O
\bigg(
(\log P)^{-\frac{1}{40} \frac{1}{(d-1)2^{d+2}}}
\frac{P^{n-d}}{(\log P)^{1/2}}
\bigg)
.\]
\end{lemma}
\begin{proof} 
Combining Lemmas~\ref{lem:easypiezy} and~\ref{lem:finredbllach}
shows that
\begin{equation}\label{eq:tofinito}
\frac{\Theta_\Q(P)}{P^{n-d}}
=
2^{1/2}
\c{C}_0 
\c{R}_1
\c{R}_2
+\c{R}_3+O\big((\log P)^{-1/2-\min\{A_1 \lambda_0,A_2/2\}}\big)
,\end{equation}
where
\[
\c{R}_1:=
\sum_{q\leq (\log P)^{A_1}}  
q^{-n}
\sum_{\substack{
\b{a}\in (\Z\cap [0,q))^2
\\
\gcd(a_1,a_2,q)=1
}}
 S_{\b{a},q} 
\overline{\mathfrak{F}(a_1,q)},\]
\[\c{R}_2:=
\int_{|\boldsymbol{\Gamma}|\leq (\log P)^{A_2}}
\frac{I(\boldsymbol{\Gamma})}{P^d}
\bigg(\int_{2}^{\max\{f_1([-1,1]^n)\}P^d}
\frac{\e(-\Gamma_1 P^{-d} t)}{\sqrt{\log t}}
\mathrm{d}t\bigg)
\mathrm{d}\boldsymbol{\Gamma} 
,\]
and
$\c{R}_3$ is a quantity that satisfies 
\[
\c{R}_3
\ll 
\sum_{q\leq (\log P)^{A_1}}  
\sum_{\substack{
\b{a}\in (\Z\cap [0,q))^2
\\
\gcd(a_1,a_2,q)=1
}}
\frac{|S_{\b{a},q}|}{q^n}
\int_{|\boldsymbol{\Gamma}|\leq (\log P)^{A_2}}
\frac{|I(\boldsymbol{\Gamma})|}{P^d}
\frac{q^3(1+|\Gamma_1|)P^d}{(\log P)^{1/2+1/7}}
\mathrm{d}\boldsymbol{\Gamma} 
.\]
We can easily see that 
\[
\c{R}_3
\ll_{A_2}
\frac{(\log P)^{3A_1+A_2}}{{(\log P)^{1/2+1/7}}}
\sum_{q\leq (\log P)^{A_1}}  
\sum_{\substack{
\b{a}\in (\Z\cap [0,q))^2
\\
\gcd(a_1,a_2,q)=1
}}
\frac{|S_{\b{a},q}|}{q^n}
\int_{|\boldsymbol{\Gamma}|\leq (\log P)^{A_2}}
|I(\boldsymbol{\Gamma})|
\mathrm{d}\boldsymbol{\Gamma} 
.\]
By~\eqref{eq:inthortobagyi} and~\eqref{eq:sumhortobagyi} 
the sum over $q$ is convergent, and so is the integral over $\boldsymbol{\Gamma}$, 
therefore 
\begin{equation}\label{eq:boundforrtria}
\c{R}_3\ll_{A_2}
(\log P)^{3 A_1+A_2-1/2-1/7}
.\end{equation}
Using~\eqref{eq:inthortobagyi}
we infer that 
\begin{eqnarray*}
& &\int_{|\boldsymbol{\Gamma}| > (\log P)^{A_2}}
\frac{|I(\boldsymbol{\Gamma})|}{P^d}
\bigg(\int_{2}^{\max\{f_1([-1,1]^n)\}P^d}
\frac{\e(-\Gamma_1 P^{-d} t)}{\sqrt{\log t}}
\mathrm{d}t\bigg)
\mathrm{d}\boldsymbol{\Gamma} 
\\ 
&\ll_{\hspace{-0,05cm}\color{white}A_2} \hspace{-0,3cm}
& \int_{|\boldsymbol{\Gamma}| > (\log P)^{A_2}}
|I(\boldsymbol{\Gamma})|
\frac{1}{\sqrt{\log P}} 
\mathrm{d}\boldsymbol{\Gamma} 
\\
&\ll_{\hspace{-0,05cm} A_2}  \hspace{-0,3cm}
& (\log P)^{-1/2-A_2/2}
,\end{eqnarray*}
therefore 
\begin{equation}\label{eq:toutreloumou}
 \c{R}_2=
\mathfrak{J}_\upphi(P)
+O_{A_2}((\log P)^{-1/2-A_2/2})
.\end{equation} 
Furthermore, by~\eqref{eq:toutreloumouerik} we deduce 
\begin{equation}\label{eq:rrenaloipon}
\c{R}_1=
\mathbb{L}_\upphi
+O_{A_1}((\log P)^{-A_1 \lambda_0})
.\end{equation} 
By
Lemmas~\ref{lem:archdensyt} and~\ref{lem:mozart_k427} 
we have 
$\mathfrak{J}_\upphi(P)\ll (\log P)^{-1/2}$, thus 
injecting
~\eqref{eq:boundforrtria},
~\eqref{eq:toutreloumou} and
~\eqref{eq:rrenaloipon} into~\eqref{eq:tofinito}
provides us with
\[\frac{\Theta_\Q(P)}{P^{n-d}}
=2^{1/2}
\c{C}_0
\mathfrak{J}_\upphi(P)
\mathbb{L}_\upphi
+O((\log P)^{-1/2-\beta}),\]
where 
$\beta:=\min \{A_1 \lambda_0,A_2/2,-3 A_1-A_2+1/7\}.$
A moment's
thought 
affirms that
assumption~\eqref{ass:Birch}
ensures the validity of 
$\lambda_0\geq (d-1)^{-1} 2^{-d-2}$
and 
choosing $A_1=\frac{1}{40}=A_2/2$ gives
$\beta\geq (40 (d-1)2^{d+2})^{-1}$.
Finally,
using Lemmas~\ref{lem:archdensyt} and~\ref{lem:mozart_k427} 
concludes the proof.
\end{proof}

\subsection{Proof of Theorem~\ref{thm:main1}} \label{proof:girise}
Define 
\begin{equation}
\label{def:theleadingconstantbwv1083}
c_\upphi :=  \frac{\mathfrak{J}}{d^{1/2}} \frac{  2^{1/2}}{\zeta(n-d)} 
\frac{\mathbb{L}_\upphi}{2} \c{C}_0.\end{equation}
By Lemmas~\ref{s:inre} and~\ref{lem:pliades}
the quantity
$N(B,\upphi,t)$
equals
\[
\frac{
\sqrt{2}
}{2}
\c{C}_0 \mathfrak{J}
\mathbb{L}_\upphi
\frac{t^{n-d}}{d^{1/2}}
\sum_{l\leq \log t} \frac{\mu(l)}{l^{n-d}(\log (t/l))^{1/2}}
\]
up to an error term 
that is 
\[
\ll  \frac{t^{n-d}}{\log t}+
\sum_{l \leq \log t} \frac{(t/l)^{n-d}}{(\log(t/l))^{\frac{1}{2}+
\frac{1}{40} \frac{1}{(d-1)2^{d+2}} 
}}
\ll \frac{t^{n-d}}{(\log t)^{\frac{1}{2}+
\frac{1}{40} \frac{1}{(d-1)2^{d+2}} 
}}
.\]
Note that for $l \leq \log t$ we have 
$(\log(t/l))^{-1/2}=(\log t)^{-1/2}+O((\log t)^{-1})$,
hence 
\[
\sum_{l\leq \log t} \frac{\mu(l)}{l^{n-d}(\log (t/l))^{1/2}}
=(\log t)^{-1/2}
\Bigg(\sum_{l\leq \log t} \frac{\mu(l)}{l^{n-d}}\Bigg)
+O((\log t)^{-1})
.\]
Assumption~\eqref{ass:Birch}
implies $n-d\geq 2$.
Denoting the Riemann zeta function by $\zeta$,
we   use
 the 
standard 
estimate 
\[\sum_{l\leq \log t} \frac{\mu(l)}{l^{n-d}}
=\zeta(n-d)^{-1}+O\bigg(\frac{1}{(\log t)^{n-d-1}}\bigg)
 \]
to obtain
\[
\sum_{l\leq \log t} \frac{\mu(l)}{l^{n-d}(\log (t/l))^{1/2}}
=
\zeta(n-d)^{-1}
(\log t)^{-1/2}
+O((\log t)^{-1})
.\] 
Thus,
\begin{equation}\label{eq:bwv565isnotbach}
\frac{N(B,\upphi,t)}
{t^{n-d}(\log t)^{-1/2}}
-
\frac{
\mathfrak{J}
\mathbb{L}_\upphi
\c{C}_0 }{
\zeta(n-d)
\sqrt{2d}
}
\ll
\frac{1}{(\log t)^{\epsilon_d}},
\end{equation} 
which concludes our proof. 
\qed  
\section{The leading constant}  
\label{s:archdens} 
The circle method and the half-dimensional sieve allowed us 
to obtain a proof of the asymptotic, however, this came at a cost because the leading constant
$c_\upphi$ in~\eqref{def:theleadingconstantbwv1083}
is complicated. In this section we shall simplify 
 $c_\upphi$ by relating
it to   a product of   $p$-adic densities.

We begin by
factorising $\mathbb{L}_\upphi$. 
One can use
 a version of the Chinese 
Remainder Theorem   to
show that  complete 
 exponential sums form a multiplicative function of the modulus. 
In the context of the circle method this is very standard 
and it occurs when one
 factorises the singular series, see~\cite[Eq.(2), section 7]{birch}, for example. Before stating the factorisation of 
$\mathbb{L}_\upphi$
we introduce the necessary notation for
the $p$-adic factors. For a prime $p$ define 
\[
\tau_{f_2}(p) :=\lim_{N\to+\infty} \frac{\#\big\{\b{t}\in (\Z\cap [0,p^N))^n:f_2(\b{t})\equiv 0\md{p^N}\big\}}{p^{N(n-1)}}
.\]
For $a\in \Z_{\geq 0}$ and
$q,k\in \N$ we   let 
\[ 
\c{W}_{a,q}(k)
:=
\sum_{\substack{ 
\ell
\in \Z \cap [0,q)
\\
\gcd(\ell,q)=\gcd(k^2,q)
}}
\e(-a\ell  / q )
\prod_{\substack{p \text{ prime } 
\\
v_p(q)>v_p(\ell)
}} 
\Big(1-\frac{1}{p}\Big)^{-1}
 \] 
and for  
 $p\equiv 3\md{4}$
  we define 
\[ 
E_\upphi(p):=
\sum_{\kappa,m \in \Z_{\geq 0}}
 \frac{\gcd(p^{2\kappa},p^m)}{p^{2\kappa+ m(n+1)}}
\sum_{\substack{
\b{a}\in (\Z\cap [0,p^m))^2
\\
\gcd(a_1,a_2,p^m)=1
}}
S_{\b{a},p^m}
 \c{W}_{a_1,p^m}(p^\kappa) 
.\] 
We furthermore define 
\[ 
E_\upphi(2):= 
\frac{1}{4}
\sum_{t,\varrho  \in \Z_{\geq 0}}
\frac{1}{2^{t+\varrho n}}
\sum_{\substack{
\b{b} \in (\Z\cap [0,2^\varrho))^2
\\ \gcd(b_1,b_2,2^\varrho)=1
}}
S_{\b{b},2^\varrho}
\e(-b_1 2^{t-\varrho})
\b{1}_{\{v_2(b_1)\geq \varrho -t -2\}}
.\]

\begin{lemma}
\label
{lem:bachcellosuiteno1}
Keep the assumptions of Theorem~\ref{thm:main1}.
Then 
\[
\mathbb{L}_\upphi
=
E_\upphi(2)
\bigg(
\prod_{p\equiv 1 \md{4}} \tau_{f_2}(p)
\bigg)
\bigg(
\prod_{p\equiv 3\md{\!4}}
E_\upphi(p)
\bigg)
,\] where both infinite products over $p$ converge absolutely. 
\end{lemma}
The proof of Lemma~\ref{lem:bachcellosuiteno1}
is based on the repeated use of explicit expressions for Ramanujan sums.
It is relatively 
straightforward but tedious and we thus omit the details. The complete 
proof is given 
in the Ph.D. thesis of the second named author~\cite[section 3.5.1]{erikthesis}.

We next 
relate the exponential sums modulo prime powers 
that the circle method gives 
to limits of counting 
functions related to $p$-adic solubility.
\begin{proposition}
\label
{prop:piisrational}
Let $p$ be a prime number with $p\equiv 3\md{4}$. 
Under the assumptions
of Theorem~\ref{thm:main1} 
the following limit exists,
\[
\ell_p:=
\lim_{N\to+\infty}
\frac{
\#\big\{\b{t}\in (\Z\cap [0,p^N))^n:
f_2(\b{t})\equiv 0 \md{p^N}, \ 
x_0^2+x_1^2=f_1(\b{t}) x_2^2
\text{ has a } \Q_p\text{-point} 
\big\}
}{p^{N(n-1)}}
.\]
Furthermore, 
we have 
$
E_\upphi(p)=
\big(1-1/p\big)^{-1} \ell_p
$. 
\end{proposition}\begin{proposition}
\label
{prop:piisprobablyrational}
Under the assumptions
of Theorem~\ref{thm:main1}, 
the following limit exists,
\[
\ell_2:=
\lim_{N\to+\infty}
\frac{
\#\big\{\b{t}\in (\Z\cap [0,2^N))^n:
f_2(\b{t})\equiv 0 \md{2^N}, \ 
x_0^2+x_1^2=f_1(\b{t}) x_2^2
\text{ has a } \Q_2\text{-point} 
\big\}
}{2^{N(n-1)}}
.\]
Furthermore,  
we have 
$ 
E_\upphi(2)=  \ell_2
$.  
\end{proposition} 
The proofs of Propositions~\ref{prop:piisrational}-\ref{prop:piisprobablyrational}
are  straightforward
in the context of the circle method 
and are not given here. Full details can be found
in~\cite[sections 3.5.2-3.5.3]{erikthesis}.

For every prime $p$ we define the number 
$$\tau_p:=\frac{(1-\frac{1}{p^{n-d}})}{(1-\frac{1}{p})}
\lim_{N\to+\infty}
\frac{
\#\big\{\b{t}\in (\Z\cap [0,p^N))^n \! : \!
p^N \mid f_2(\b{t}),
x_0^2+x_1^2=f_1(\b{t}) x_2^2
\text{ has a } \Q_p\text{-point} 
\big\}
}{p^{N(n-1)}}
.$$
This is well-defined because for $p\equiv 1 \md{4}$ the limit coincides with $\tau_{f_2}(p)$ 
and for $p\not\equiv 1\md{4}$ the limit coincides with $\ell_p$ and $\ell_2$.
The definition of $\tau_p$ is motivated by the construction of the Tamagawa measure by Loughran
in~\cite[\S 5.7.2]{loughranjems}. It is useful to recall that if one was counting $\Q$-rational points on the hypersurface $f_2=0$ then the corresponding Peyre constant would involve a $p$-adic density that is the same as the number $\tau_p$ except for the condition on 
$\Q_p$-solubility, see~\cite[Cor.3.5]{MR1681100}. For  $s\in \mathbb{C}$ with $\Re(s)>1$ let \begin{equation}
\label{def:dansuggestion}L(s):=\sqrt{\zeta(s)},\end{equation} denote the $p$-adic 
factor of $L(s)$ by $L_p(s)$ and write $\lambda_p$ for $L_p(1)$,  i.e.,$$\lambda_p:=\bigg(1-\frac{1}{p}\bigg)^{-1/2}.$$ Recall the 
definition of the real density $\mathfrak{J}$  in~\eqref{def:mintokratas} and that $d$ denotes the degrees of $f_1$ and $f_2$ (which are equal by the assumption of Theorem~\ref{thm:main1}). \begin{theorem}\label{thm:main2} Keep the assumptions of Theorem~\ref{thm:main1}. 
\begin{enumerate} 
\item   If $\upphi$ has a smooth fibre with a $\Q$-point then the constant  $c_\upphi$ in Theorem~\ref{thm:main1} is strictly positive.
\item  The infinite product $\prod_{p}\frac{\tau_p}{\lambda_p}$ taken over all non-archimedean places converges. 
\item  The constant  $c_\upphi$ in Theorem~\ref{thm:main1} satisfies 
\[c_\upphi= \frac{\frac{1}{\sqrt{d}} \mathfrak{J} \prod_{p}\frac{\tau_p}{\lambda_p}}{\sqrt{\pi}}.\]
\end{enumerate}
\end{theorem}\begin{remark} Recalling that 
$\sqrt{\pi}$ is the value of the Euler Gamma function at $1/2$ and  noting that  \[1=\lim_{s\to 1_+} (s-1)^{1/2} L(s) \]
allows for a
comparison of Theorem~\ref{thm:main2}
with the case 
of~\cite[Th. 5.15]{loughranjems}
that corresponds to 
\[
\rho_\c{B}(X)=\frac{1}{2}
.\] 
\end{remark}
\begin{proof}[Proof of Theorem~\ref{thm:main2}] 
To prove~\emph{(1)} observe that due to~\eqref{def:theleadingconstantbwv1083},
it suffices to show that if $\upphi$ has a smooth fibre with a $\Q$-point  then 
\[\mathfrak{J}>0 \ \text{ and } \ 
\mathbb{L}_\upphi>0
.\]

For the former part, we recall that 
it is standard that if $\c{B}\subset [-1,1]^n$ is a box with sides parallel to the coordinate axes 
and  the hypersurface $f_2=0$ has a non-singular real point inside $\c{B}$
then the corresponding singular integral that is given by
\[
\int_{\Gamma \in \R}
\int_{\substack{
\b{t} \in  \c{B}
}} \e(\Gamma f_2(\b{t})) 
\mathrm{d}\b{t}
\mathrm{d}\Gamma 
\]
 is strictly positive. This is proved in~\cite[\S 6]{birch}, for example, but see also~\cite[\S 4]{MR693325}.
Here, the fact that $\upphi$ has a smooth fibre with a $\Q$-point implies that
there exists $b\in \P^n(\Q)$ such that $f_2(b)=0$ and the curve $x_0^2+x_1^2=f_1(\b{t}) x_2^2$ is smooth and
has a $\Q$-point, hence in particular, an $\R$-point. Picking $\b{t}_0\in \Z_\text{prim}^n$ with $b=[\b{t}_0]$ we get that 
there exists $\b{t}_0  \in \R^n$  with $f_2(\b{t}_0)=0$ and  $f_1(\b{t}_0) >0$.
Note that $f_2$ is smooth at  $\b{t}_0$  due to the assumptions of Theorem~\ref{thm:main1}.
Thus, by the Implicit Function Theorem  there is a non-empty box  $\c{B}$ with sides parallel to the axes such that every $\b{t}$ with $f_2(\b{t})=0$ and 
in the interior of $\c{B}$ satisfies $f_1(\b{t})>0$. From this, one infers that $\mathfrak{J}>0$ upon recalling the definition of $\mathfrak{J}$ 
in~\eqref{def:mintokratas}.

To prove that $\mathbb{L}_\upphi>0$, we invoke Lemma~\ref{lem:bachcellosuiteno1} 
to see
that it is enough to show 
\begin{equation}
\label{eq:bachbwv552}
E_\upphi(2)>0, \  \ 
p\equiv 1 \md{4} \Rightarrow \tau_{f_2}(p)>0
\ \text{ and } \ 
p\equiv 3\md{\!4} \Rightarrow
E_\upphi(p) >0
.\end{equation}
For this, note that for every prime $p$ the point $\b{t}_0$   can be viewed as a smooth $\Q_p$-point on the hypersurface
$f_2 =0$ and such that the curve
$x_0^2+x_1^2=f_1(\b{t}_0)x_2^2$ has a point $\Q_p$-point. If $p\equiv 1\md{4}$ this forces no condition on $f_1(\b{t}_0)$,
thus $\tau_{f_2}(p)>0$ because, as mentioned in~\cite[\S 7]{birch}, one can use Hensel's lemma to
prove that if $f_2=0$ has a smooth $\Q_p$-point then the analogous $p$-adic density is strictly positive. 
If $p\equiv 3\md{4}$ or if $p=2$ then the existence of such a $\b{t}_0$ can be used with 
Hensel's lemma 
to prove that the quantities $\ell_2$ and $\ell_p$ 
 are strictly positive. The equalities 
$E_\upphi(p)= \ell_p/(1-1/p)$
and $E_\upphi(2)=  \ell_2$ (proved in Propositions~\ref{prop:piisrational}-\ref{prop:piisprobablyrational}) 
then show the validity of~\eqref{eq:bachbwv552},
which concludes the proof of \emph{(1)}.

Let us now commence the proof of \emph{(2)}.
Denoting the limit in the definition of $\tau_p$ by $\ell_p$ we see that 
\begin{align*}
\lim_{t\to+\infty} \prod_{p\leq t} \frac{\tau_p}{\lambda_p}
=&\lim_{t\to+\infty}\prod_{p\leq t} \frac{(1-\frac{1}{p^{n-d}})}{(1-\frac{1}{p})} \ell_p  \bigg(1-\frac{1}{p}\bigg)^{1/2} \\
=&\frac{ \ell_2 2^{1/2}}{\zeta(n-d)} \lim_{t\to+\infty}
\prod_{p\leq t}  \frac{\ell_p}{(1-\frac{\b{1}_{p\equiv 3\md{4}}}{p})} 
\Bigg( \frac{(1-\frac{\b{1}_{p\equiv 3\md{4}}}{p})}{(1-\frac{\b{1}_{p\equiv 1\md{4}}}{p})}\Bigg)^{1/2}
.\end{align*}
We now let $\chi$ stand for 
the non-trivial Dirichlet character $\md{4}$ to obtain that 
\[
\prod_{p\leq t}
\frac{(1-\frac{\b{1}_{p\equiv 3\md{4}}}{p})}{(1-\frac{\b{1}_{p\equiv 1\md{4}}}{p})}=
\Bigg(\prod_{p\leq t} \frac{1}{1-\frac{\chi(p)}{p}} \Bigg) 
 \prod_{\substack{p\leq t \\ p\equiv 3\md{4} }}   \Big(1-\frac{1}{p^2}\Big) 
\]
and therefore, alluding to the well-known fact that the Euler product for the Dirichlet  series $L(s,\chi)$ of $\chi$ 
converges  to $\pi/4$ for $s=1$, we get via Definition~\eqref{def:c_0} that 
\[\lim_{t\to+\infty} \prod_{p\leq t}  
\Bigg( \frac{(1-\frac{\b{1}_{p\equiv 3\md{4}}}{p})}{(1-\frac{\b{1}_{p\equiv 1\md{4}}}{p})}\Bigg)^{1/2}
=\frac{\pi^{1/2}}{2} \c{C}_0
.\]
We have so far shown that 
\[
\lim_{t\to+\infty} \prod_{p\leq t} \frac{\tau_p}{\lambda_p}
=
\frac{ \ell_2 2^{1/2}}{\zeta(n-d)} \Bigg(\lim_{t\to+\infty}
\prod_{p\leq t}  \frac{\ell_p}{(1-\frac{\b{1}_{p\equiv 3\md{4}}}{p})} 
\Bigg)\frac{\pi^{1/2}}{2} \c{C}_0
.\]
It is clear that if $p\equiv 1 \md{4}$ then $\ell_p=\tau_{f_2}(p)$, and thus, 
\[\lim_{t\to+\infty}\prod_{\substack{p\equiv 1 \md{4}\\p\leq t}} \ell_p=
\prod_{p\equiv 1 \md{4}} \tau_{f_2}(p)
.\]
By  
Proposition~\ref{prop:piisrational}
one gets 
\[  \prod_{\substack{p\equiv 3 \md{4} \\ p\leq t }}\frac{\ell_p}{(1-\frac{1}{p})}= 
\prod_{\substack{p\equiv 3 \md{4}\\ p\leq t }} E_\upphi(p).\]
It is now clear from Lemma~\ref{lem:bachcellosuiteno1} that the last product converges as $t\to+\infty$, 
therefore the product $\prod_{p} \tau_p/\lambda_p$
is  convergent, which proves \emph{(2)}.

For the proof of \emph{(3)} we note that the arguments at the end of the proof of \emph{(2)} provided us with the equality 
\[\prod_p \frac{\tau_p}{\lambda_p}=
\frac{ \ell_2 2^{1/2}}{\zeta(n-d)}
\Bigg(\prod_{p\equiv 1 \md{4}}\tau_{f_2}(p)\Bigg) 
\Bigg(\prod_{\substack{p\equiv 3 \md{4}}} E_\upphi(p)\Bigg) 
\frac{\pi^{1/2}}{2} \c{C}_0 .\] We have $E_\upphi(2)=  \ell_2$ due to Proposition~\ref{prop:piisprobablyrational}, and alluding to 
Lemma~\ref{lem:bachcellosuiteno1} we get 
\[
\prod_{p} \frac{\tau_p}{\lambda_p}=\frac{  2^{1/2}}{\zeta(n-d)}
\mathbb{L}_\upphi
\frac{\pi^{1/2}}{2} \c{C}_0 .
\]
A comparison with~\eqref{def:theleadingconstantbwv1083} makes the proof of \emph{(3)} immediately apparent. 
\end{proof}
Let us remark that the arguments in the present section 
can be easily rearranged to show that 
$\prod_{p\leq t}\tau_p$ diverges and therefore, the numbers $\lambda_p$ can be viewed as `convergence factors'.
We are very grateful to Daniel Loughran for suggesting this choice for $\lambda_p$, as well as for 
the $L$-function in~\eqref{def:dansuggestion}.

\end{document}